\documentclass[reqno]{amsart}

\usepackage{mathabx} 
\usepackage{graphicx}
\usepackage{amssymb,amsmath,amsthm,amscd}
\usepackage{mathrsfs}
\usepackage{enumerate}
\usepackage{enumitem}
\usepackage{verbatim}
\usepackage[usenames,dvipsnames]{color}
\usepackage[colorlinks=true, pdfstartview=FitV,
linkcolor=blue,citecolor=blue,urlcolor=blue]{hyperref}
\usepackage[all]{xy}
\usepackage{tikz}
\usepackage{tikz-cd}
\usetikzlibrary{matrix, arrows, decorations.pathmorphing}
\usepackage{stmaryrd} 
\usepackage{dsfont}
\usepackage{bm} 
\usepackage{array}
\usepackage{adjustbox}

\allowdisplaybreaks[4]

\DeclareSymbolFont{sansops}{OT1}{\sfdefault}{m}{n}
\SetSymbolFont{sansops}{bold}{OT1}{\sfdefault}{b}{n}

\makeatletter
\renewcommand\operator@font{\mathgroup\symsansops}
\makeatother

\newcommand{\nc}{\newcommand}
\newcommand{\rnc}{\renewcommand}

\numberwithin{equation}{section}

\nc{\eqrefs}[2]{\text{(\ref{#1}-\ref{#2})}}

\usepackage[colorinlistoftodos]{todonotes}
\usepackage{mathtools}
\mathtoolsset{showonlyrefs,showmanualtags}
\usepackage{empheq}

\theoremstyle{plain}
\newtheorem{lemma}{Lemma}[subsection]
\newtheorem{prop}[lemma]{Proposition}
\newtheorem{theorem}[lemma]{Theorem}

\newcommand{\Prop}{\begin{prop}}
	\newcommand{\enprop}{\end{prop}}
\newcommand{\Lemma}{\begin{lemma}}
	\newcommand{\enlemma}{\end{lemma}}
\newcommand{\Th}{\begin{theorem}}
	\newcommand{\enth}{\end{theorem}}
\newtheorem{corollary}[lemma]{Corollary}
\newcommand{\Cor}{\begin{corollary}}
	\newcommand{\encor}{\end{corollary}}
\newtheorem{definition}[lemma]{Definition}
\newtheorem{conjecture}[lemma]{Conjecture}
\newcommand{\Def}{\begin{definition}}
	\newcommand{\edf}{\end{definition}}
\newtheorem{sublemma}[lemma]{Sublemma}
\newcommand{\Sublemma}{\begin{sublemma}}
	\newcommand{\ensub}{\end{sublemma}}

\theoremstyle{definition}

\newtheorem{remark}[lemma]{Remark}
\newtheorem{remarks}[lemma]{Remarks}
\newtheorem{example}[lemma]{Example}
\newtheorem{Convention}[lemma]{Convention}
\newcommand{\Conv}{\begin{Convention}}
	\newcommand{\enconv}{\end{Convention}}
\nc{\Rem}{\begin{remark}}
	\nc{\enrem}{\end{remark}}

\nc{\rmkend}{\hfill \ensuremath{\triangledown}}
\nc{\defend}{\hfill \ensuremath{\triangle}}

\nc{\be}{\begin{enumerate}}
	\nc{\ee}{\end{enumerate}}
\newcommand{\eq}{\begin{eqnarray}}
	\newcommand{\eneq}{\end{eqnarray}}
\nc{\bc}{\begin{cases}}
	\nc{\ec}{\end{cases}}
\newcommand{\eqn}{\begin{eqnarray*}}
	\newcommand{\eneqn}{\end{eqnarray*}}
\newcommand{\ba}{\begin{array}}
	\newcommand{\ea}{\end{array}}

\newcommand{\arxiv}[1]{\href{http://arxiv.org/abs/#1}{\tt arXiv:\nolinkurl{#1}}}

\nc{\cA}{{\mathcal A}}
\nc{\cB}{{\mathcal B}}
\nc{\cC}{{\mathcal C}}
\nc{\cD}{{\mathcal D}}
\nc{\cE}{{\mathcal E}}
\nc{\cF}{{\mathcal F}}
\nc{\cG}{{\mathcal G}}
\nc{\cH}{{\mathcal H}}
\nc{\cI}{{\mathcal I}}
\nc{\cJ}{{\mathcal J}}
\nc{\cK}{{\mathcal K}}
\nc{\cL}{{\mathcal L}}
\nc{\cM}{\mathcal{M}}
\nc{\cN}{{\mathcal N}}
\nc{\cO}{{\mathcal O}}
\nc{\cP}{{\mathcal P}}
\nc{\calQ}{{\mathcal Q}}
\nc{\cR}{{\mathcal R}}
\nc{\cS}{\mathcal{S}}
\nc{\cT}{{\mathcal T}}
\nc{\cU}{\mathcal{U}}
\nc{\cV}{{\mathcal V}}
\nc{\cX}{{\mathcal X}}
\nc{\cY}{\mathcal{Y}}
\nc{\cW}{\mathcal{W}}
\nc{\cZ}{{\mathcal Z}}

\nc{\bbA}{{\mathbb{A}}}
\nc{\bbB}{{\mathbb{B}}}
\nc{\bbC}{{\mathbb{C}}}
\nc{\bbD}{{\mathbb{D}}}
\nc{\bbE}{{\mathbb{E}}}
\nc{\bbF}{{\mathbb{F}}}
\nc{\bbG}{{\mathbb{G}}}
\nc{\bbH}{{\mathbb{H}}}
\nc{\bbI}{{\mathbb{I}}}
\nc{\bbJ}{{\mathbb{J}}}
\nc{\bbK}{{\mathbb{K}}}
\nc{\bbL}{{\mathbb{L}}}
\nc{\bbM}{{\mathbb{M}}}
\nc{\bbN}{{\mathbb{N}}}
\nc{\bbO}{{\mathbb{O}}}
\nc{\bbP}{{\mathbb{P}}}
\nc{\bbQ}{{\mathbb{Q}}}
\nc{\bbR}{{\mathbb{R}}}
\nc{\bbS}{{\mathbb{S}}}
\nc{\bbT}{{\mathbb{T}}}
\nc{\bbU}{{\mathbb{U}}}
\nc{\bbV}{{\mathbb{V}}}
\nc{\bbX}{{\mathbb{X}}}
\nc{\bbY}{{\mathbb{Y}}}
\nc{\bbW}{{\mathbb{W}}}
\nc{\bbZ}{{\mathbb{Z}}}

\usepackage{mathrsfs}

\nc{\scrA}{{\mathscr A}}
\nc{\scrB}{{\mathscr B}}
\nc{\scrC}{{\mathscr C}}
\nc{\scrD}{{\mathscr D}}
\nc{\scrE}{{\mathscr E}}
\nc{\scrF}{{\mathscr F}}
\nc{\scrG}{{\mathscr G}}
\nc{\scrH}{{\mathscr H}}
\nc{\scrI}{{\mathscr I}}
\nc{\scrJ}{{\mathscr J}}
\nc{\scrK}{{\mathscr K}}
\nc{\scrL}{{\mathscr L}}
\nc{\scrM}{{\mathscr M}}
\nc{\scrN}{{\mathscr N}}
\nc{\scrO}{{\mathscr O}}
\nc{\scrP}{{\mathscr P}}
\nc{\scrQ}{{\mathscr Q}}
\nc{\scrR}{{\mathscr R}}
\nc{\scrS}{{\mathscr S}}
\nc{\scrT}{{\mathscr T}}
\nc{\scrU}{{\mathscr U}}
\nc{\scrV}{{\mathscr V}}
\nc{\scrX}{{\mathscr X}}
\nc{\scrY}{{\mathscr Y}}
\nc{\scrW}{{\mathscr W}}
\nc{\scrZ}{{\mathscr Z}}

\nc{\sfA}{{\mathsf A}}
\nc{\sfB}{{\mathsf B}}
\nc{\sfC}{{\mathsf C}}
\nc{\sfD}{{\mathsf D}}
\nc{\sfE}{{\mathsf E}}
\nc{\sfF}{{\mathsf F}}
\nc{\sfG}{{\mathsf G}}
\nc{\sfH}{{\mathsf H}}
\nc{\sfI}{{\mathsf I}}
\nc{\sfJ}{{\mathsf J}}
\nc{\sfK}{{\mathsf K}}
\nc{\sfL}{{\mathsf L}}
\nc{\sfM}{{\mathsf M}}
\nc{\sfN}{{\mathsf N}}
\nc{\sfO}{{\mathsf O}}
\nc{\sfP}{{\mathsf P}}
\nc{\sfQ}{{\mathsf Q}}
\nc{\sfR}{{\mathsf R}}
\nc{\sfS}{{\mathsf S}}
\nc{\sfT}{{\mathsf T}}
\nc{\sfU}{{\mathsf U}}
\nc{\sfV}{{\mathsf V}}
\nc{\sfX}{{\mathsf X}}
\nc{\sfY}{{\mathsf Y}}
\nc{\sfW}{{\mathsf W}}
\nc{\sfZ}{{\mathsf Z}}

\nc{\sfa}{{\mathsf a}}
\nc{\sfb}{{\mathsf b}}
\nc{\sfc}{{\mathsf c}}
\nc{\sfd}{{\mathsf d}}
\nc{\sfe}{{\mathsf e}}
\nc{\sff}{{\mathsf f}}
\nc{\sfg}{{\mathsf g}}
\nc{\sfh}{{\mathsf h}}
\nc{\sfi}{{\mathsf i}}
\nc{\sfj}{{\mathsf j}}
\nc{\sfk}{{\mathsf k}}
\nc{\sfl}{{\mathsf l}}
\nc{\sfm}{{\mathsf m}}
\nc{\sfn}{{\mathsf n}}
\nc{\sfo}{{\mathsf o}}
\nc{\sfp}{{\mathsf p}}
\nc{\sfq}{{\mathsf q}}
\nc{\sfr}{{\mathsf r}}
\nc{\sfs}{{\mathsf s}}
\nc{\sft}{{\mathsf t}}
\nc{\sfu}{{\mathsf u}}
\nc{\sfv}{{\mathsf v}}
\nc{\sfx}{{\mathsf x}}
\nc{\sfy}{{\mathsf y}}
\nc{\sfw}{{\mathsf w}}
\nc{\sfz}{{\mathsf z}}

\nc {\bfA}{{\mathbf A}}
\nc {\bfB}{{\mathbf B}}
\nc {\bfC}{{\mathbf C}}
\nc {\bfD}{{\mathbf D}}
\nc {\bfE}{{\mathbf E}}
\nc {\bfF}{{\mathbf F}}
\nc {\bfG}{{\mathbf G}}
\nc {\bfH}{{\mathbf H}}
\nc {\bfI}{{\mathbf I}}
\nc {\bfJ}{{\mathbf J}}
\nc {\bfK}{{\mathbf K}}
\nc {\bfL}{{\mathbf L}}
\nc {\bfM}{{\mathbf M}}
\nc {\bfN}{{\mathbf N}}
\nc{\bfO}{{\mathbf O}}
\nc {\bfP}{{\mathbf P}}
\nc {\bfQ}{{\mathbf Q}}
\nc {\bfR}{{\mathbf R}}
\nc {\bfS}{{\mathbf S}}
\nc {\bfT}{{\mathbf T}}
\nc {\bfU}{{\mathbf U}}
\nc {\bfV}{{\mathbf V}}
\nc {\bfX}{{\mathbf X}}
\nc {\bfY}{{\mathbf Y}}
\nc {\bfW}{{\mathbf W}}
\nc {\bfZ}{{\mathbf Z}}

\nc {\fka}{{\mathfrak a}}
\nc {\fkb}{{\mathfrak b}}
\nc {\fkc}{{\mathfrak c}}
\nc {\fkd}{{\mathfrak d}}
\nc {\fke}{{\mathfrak e}}
\nc {\fkf}{{\mathfrak f}}
\nc {\fkg}{{\mathfrak g}}
\nc {\fkh}{{\mathfrak h}}
\nc {\fki}{{\mathfrak i}}
\nc {\fkj}{{\mathfrak j}}
\nc {\fkk}{{\mathfrak k}}
\nc {\fkl}{{\mathfrak l}}
\nc {\fkm}{{\mathfrak m}}
\nc {\fkn}{{\mathfrak n}}
\nc {\fko}{{\mathfrak o}}
\nc {\fkp}{{\mathfrak p}}
\nc {\fkq}{{\mathfrak q}}
\nc {\fkr}{{\mathfrak r}}
\nc {\fks}{{\mathfrak s}}
\nc {\fkt}{{\mathfrak t}}
\nc {\fku}{{\mathfrak u}}
\nc {\fkv}{{\mathfrak v}}
\nc {\fkx}{{\mathfrak x}}
\nc {\fky}{{\mathfrak y}}
\nc {\fkw}{{\mathfrak w}}
\nc {\fkz}{{\mathfrak z}}

\nc {\fkA}{{\mathfrak A}}
\nc {\fkB}{{\mathfrak B}}
\nc {\fkC}{{\mathfrak C}}
\nc {\fkD}{{\mathfrak D}}
\nc {\fkE}{{\mathfrak E}}
\nc {\fkF}{{\mathfrak F}}
\nc {\fkG}{{\mathfrak G}}
\nc {\fkH}{{\mathfrak H}}
\nc {\fkI}{{\mathfrak I}}
\nc {\fkJ}{{\mathfrak J}}
\nc {\fkK}{{\mathfrak K}}
\nc {\fkL}{{\mathfrak L}}
\nc {\fkM}{{\mathfrak M}}
\nc {\fkN}{{\mathfrak N}}
\nc {\fkO}{{\mathfrak O}}
\nc {\fkP}{{\mathfrak P}}
\nc {\fkQ}{{\mathfrak Q}}
\nc {\fkR}{{\mathfrak R}}
\nc {\fkS}{{\mathfrak S}}
\nc {\fkT}{{\mathfrak T}}
\nc {\fkU}{{\mathfrak U}}
\nc {\fkV}{{\mathfrak V}}
\nc {\fkX}{{\mathfrak X}}
\nc {\fkY}{{\mathfrak Y}}
\nc {\fkW}{{\mathfrak W}}
\nc {\fkZ}{{\mathfrak Z}}

\rnc{\a}{\fka}
\rnc{\b}{\fkb}
\rnc{\c}{\fkc}
\rnc{\d}{\fkd}
\nc{\e}{\fke}
\nc{\f}{\fkf}
\nc{\g}{\fkg}
\nc{\h}{\fkh}
\rnc{\i}{\fki}
\rnc{\j}{\fkj}
\rnc{\k}{\fkk}
\rnc{\l}{\fkl}
\nc{\m}{\fkm}
\nc{\n}{\fkn}
\rnc{\o}{\fko}
\nc{\p}{\fkp}
\nc{\q}{\fkq}
\rnc{\r}{\fkr}
\nc{\s}{\fks}
\rnc{\t}{\fkt}
\rnc{\u}{\fku}
\rnc{\v}{\fkv}
\nc{\x}{\fkx}
\nc{\y}{\fky}
\nc{\w}{\fkw}
\nc{\z}{\fkz}

\newcommand{\C}{{\mathbb C}}
\newcommand{\Q}{\mathbb {Q}}
\newcommand{\Z}{{\mathbb Z}}

\nc{\Ad}{\operatorname{Ad}}
\nc{\ad}{\operatorname{ad}}
\nc{\sym}{\mathfrak{S}} 
\nc{\weyl}{\mathfrak{W}}
\nc{\Serre}{\operatorname{Serre}}

\nc{\gfin}{\mathring{\g}}
\nc{\hfin}{\mathring{\h}}
\nc{\Dfin}{\mathring{D}}
\nc{\Pfin}{\mathring{\Plat}}
\nc{\Qfin}{\mathring{\Qlat}}
\nc{\rhofin}{\mathring{\rho}}
\nc{\Ifin}{\mathring{I}}

\newcommand{\Hom}{\operatorname{Hom}}
\newcommand{\End}{\operatorname{End}}
\nc{\Aut}{\operatorname{Aut}}

\nc{\coker}{{\operatorname{coker}}}

\newcommand{\id}{{\operatorname{id}}}
\nc{\Id}{\operatorname{Id}}
\nc{\aff}{{\sf aff}}
\nc{\Sp}{\operatorname{Sp}}

\nc{\cork}{{\operatorname{cork}}} 
\nc{\rank}{{\operatorname{rank}}}
\nc{\Rep}{{\operatorname{Rep}}}
\nc{\Repfd}{{\operatorname{Rep}_{\sf fd}}}

\nc{\Mod}{{\operatorname{Mod}}}
\nc{\Modfd}{{\operatorname{Mod}_{\sf fd}}}

\newcommand{\Tr}{\mathop{\sf Tr}}

\renewcommand{\t}{\mathsf{t}}

\nc {\ul}{\underline}
\nc {\ol}{\overline}
\nc {\wtil}{\widetilde}
\nc {\wh}{\widehat}
\nc {\wb}{\widebar}
\nc{\scs}{\scriptscriptstyle}
\nc{\scsop}{\scriptscriptstyle\operatorname} 

\nc {\ie}{\emph{i.e.}, } 
\nc {\eg}{\emph{e.g.}, } 
\nc {\cf}{{\emph{cf.}} } 
\nc {\loccit}{{\emph{loc. cit. }} } 

\nc {\aand}{\qquad\mbox{and}\qquad}

\nc{\al}{\alpha}
\nc{\ga}{\gamma}
\nc{\del}{\delta}
\nc{\eps}{\epsilon}
\nc{\ze}{\zeta}
\nc{\ka}{\kappa}
\nc{\la}{\lambda}
\nc{\si}{\sigma}
\nc{\om}{\omega}
\nc{\ups}{\upsilon}

\nc{\Del}{\Delta}
\nc{\La}{\Lambda}
\nc{\Ups}{\Upsilon}

\nc{\ot}{\otimes}

\nc{\ext}{\mathsf{ext}}

\nc{\fksl}{\mathfrak{sl}}
\nc{\fkso}{\mathfrak{so}}
\nc{\fksp}{\mathfrak{sp}}

\nc{\qu}{\quad}
\nc{\qq}{\qquad}

\rnc{\eq}[1]{\begin{equation} #1 \end{equation}}

\nc{\drc}[1]{\delta_{#1}}
\nc{\der}{\partial}
\nc{\sfad}{\mathsf{ad}}
\nc{\ten}{\otimes}

\nc{\ourcomment}[1]{{\color{teal}{//#1//}}}
\nc{\andreacomment}[1]{{\color{magenta}{//Andrea: #1//}}}
\nc{\bartcomment}[1]{{\color{red}{//Bart: #1//}}}
\nc{\Omit}[1]{}
\nc{\summary}[1]{}
\nc{\tm}[1]{{\color{magenta}{#1}}}

\nc{\bsF}{\bbF} 

\nc{\Ons}{\mathbf{O}_q}
\nc{\aOns}{\mathbf{O}_q^{\mathsf{a}}}
\nc{\bOns}{\mathbf{O}_q^{\mathsf{b}}}

\nc{\km}{\mathbf{k}}
\nc{\qkm}{\mathfrak{k}}

\nc{\texp}[1]{\wt{s}_{#1}} 
\nc{\Lus}[1]{T_{#1}} 
\nc{\mLus}[1]{\mathbf{T}_{#1}} 

\nc{\fin}{{0}}

\nc{\Afin}{{A_\fin}}

\rnc{\aff}[1]{\wh{#1}}

\nc{\de}[1]{\epsilon_{#1}} 

\nc{\fIS}{I} 
\nc{\aIS}{\aff{I}} 
\nc{\IS}{I}

\nc{\veps}{\varepsilon}
\nc{\vtheta}{\vartheta}

\nc{\rt}[1]{\alpha_{#1}} 
\nc{\cort}[1]{h_{#1}}
\nc{\fwt}[1]{\varpi_{#1}}
\nc{\fcwt}[1]{\Lambda_{#1}^\vee}

\nc{\rootsys}{{\bm\Phi}} 

\nc{\hrt}{\vartheta} 

\nc{\cp}[2]{#1(#2)}
\nc{\iip}[2]{\left( #1 , #2\right)} 

\nc{\drv}[1]{\delta_{#1}} 
\nc{\codrv}[1]{d_{#1}} 

\nc{\bsfld}{\mathbb{K}} 

\nc{\central}[1]{c_{#1}}

\nc{\Qlat}{\mathsf{Q}}
\nc{\Qpm}{\Qlat^\pm}
\nc{\Qp}{\Qlat^+}
\nc{\Qm}{\Qlat^-}
\nc{\Qlatv}{\Qlat^\vee}
\nc{\Qvpm}{\Qlat^{\vee,\pm}}
\nc{\Qvp}{\Qlat^{\vee,+}}
\nc{\Qvm}{\Qlat^{\vee,-}}
\nc{\Qextp}{\Qlat^+_{\sf ext}}
\nc{\Qvext}{\Qlat^{\vee}_{\sf ext}}
\nc{\Qvextt}{\Qlat^{\vee}_{{\sf ext},\tau}}
\nc{\Qvextp}{\Qlat^{\vee,+}_{\sf ext}}
\nc{\aQ}{\aff{\Qlat}}
\nc{\aQv}{\aff{\Qlat}^{\vee}}
\nc{\aQp}{\aff{\Qlat}^+}
\nc{\aQextp}{\aff{\Qlat}^+_{\sf ext}}
\nc{\aQvp}{\aff{\Qlat}^{\vee,+}}
\nc{\aQvext}{\wt{\Qlat}^{\vee}} 
\nc{\aQvextt}{\aff{\Qlat}^{\vee}_{{\sf ext},\tau}}
\nc{\aQvextp}{\aff{\Qlat}^{\vee,+}_{\sf ext}}
\nc{\Plat}{\mathsf{P}}
\nc{\Platv}{\mathsf{P}^\vee}
\nc{\Ppm}{\mathsf{P}^\pm}
\nc{\Pp}{\mathsf{P}^+}
\nc{\Pm}{\mathsf{P}^-}
\nc{\Pvpm}{\mathsf{P}^{\vee,\pm}}
\nc{\Pvp}{\mathsf{P}^{\vee,+}}
\nc{\Pvm}{\mathsf{P}^{\vee,-}}
\nc{\aP}{\wh{\Plat}}
\nc{\Pext}{\Plat}
\nc{\Pextpm}{\Plat_{\sf ext}^{\pm}}
\nc{\aPext}{\wh{\Plat}_{{\sf ext}}}
\nc{\aPpm}{\aP^{\pm}}
\nc{\aPextpm}{\aP^{\pm}_{\sf ext}}
\nc{\aPd}{{\aP/\bbZ\drv{}}} 

\nc{\PZ}{{\Plat_{\Z}}}
\nc{\PvZ}{{\Platv_{\Z}}}

\nc{\cl}[1]{\operatorname{cl}(#1)}

\nc{\Pcl}{\Plat_{\cl}}
\nc{\Pvcl}{\Plat^{\vee}_{\cl}}

\nc{\Pclz}{\Plat_{\cl,0}}

\nc{\fr}{_\mathsf{fr}}
\nc{\Qfr}{\Qlat\fr}
\nc{\Pfr}{\Plat\fr}

\nc{\Qvfr}{\Qlatv\fr}
\nc{\Pvfr}{\Platv\fr}

\nc{\wgt}[1]{\operatorname{wt}(#1)}

\nc{\hp}{\h'}
\nc{\gp}{\g'}
\nc{\kp}{\k'}

\nc{\Wfin}{\mathring{W}}
\nc{\Waff}{W}
\nc{\Wext}{\widetilde{W}}

\nc{\rfl}[1]{{s_{#1}}}

\nc{\UqgKM}{U_q(\g_{\scsop{KM}})}
\nc{\gKM}{\g_{\scsop{KM}}}

\nc{\Uqg}{U_q(\g)} 
\nc{\Uqgp}{U_q(\g')} 

\nc{\Uqh}{U_q(\h)} 
\nc{\Uqhp}{U_q(\h')} 

\nc{\Uqb}{U_q{(\b)}} 
\nc{\Uqbp}{U_q{(\b^+)}} 
\nc{\Uqbm}{U_q{(\b^-)}} 
\nc{\Uqbpm}{U_q{({\b}^{\pm})}} 
\nc{\Uqn}{U_q{(\n)}} 
\nc{\Uqnm}{U_q{({\n}^-)}} 
\nc{\Uqnp}{U_q{({\n}^+)}} 
\nc{\Uqnpm}{U_q{({\n}^{\pm})}} 

\nc{\Uqag}{U_q(\wt{\g})} 
\nc{\Uqagp}{U_q(\wh{\g})} 

\nc{\Uqan}{U_q(\wh{\n})} 
\nc{\Uqanp}{U_q(\wh{\n}^+)} 
\nc{\Uqanm}{U_q(\wh{\n}^-)} 
\nc{\Uqanpm}{U_q(\wh{\n}^{\pm})} 

\nc{\Uqah}{U_q(\wt{\h})} 
\nc{\Uqahp}{U_q(\wh{\h})} 

\nc{\Uqab}{U_q(\wt{\b})} 
\nc{\Uqabp}{U_q(\wt{\b}^+)} 
\nc{\Uqabm}{U_q(\wt{\b}^-)} 
\nc{\Uqabpm}{U_q(\wt{\b}^{\pm})} 
\nc{\Uqabpp}{U_q(\wh{\b}^+)} 

\nc{\UqgX}{U_q(\g_X)} 
\nc{\UqgZ}{U_q(\g_Z)}
\nc{\UqbpX}{U_q(\b_X^+)}
\nc{\UqnmX}{U_q(\n_X^-)}
\nc{\Uqht}{U_q(\h^{\theta})}

\nc{\CUqag}[1]{\Uqag^{#1}^{\cO^{\sf int}}}

\nc{\CUqg}[1]{\Uqg^{#1}^{\cO^{\sf int}}}

\nc{\Lg}{\mathscr{L}}
\nc{\ag}{\wt{\g}} 
\nc{\agp}{\wh{\g}} 
\nc{\ah}{\wt{\h}} 
\nc{\ahp}{\wh{\h}} 

\nc{\UqLg}{U_q(\Lg)} 

\nc{\CUqLg}[1]{(\UqLg^{#1})^{\sf fd}} 

\nc{\Kg}[1]{k_{#1}}
\nc{\Kgpm}[1]{k_{#1}^{\pm}}
\nc{\Kgp}[1]{k_{#1}^+}
\nc{\Kgm}[1]{k_{#1}^-}
\nc{\Eg}[1]{E_{#1}}
\nc{\Fg}[1]{F_{#1}}
\nc{\egp}[1]{e^{+}_{#1}}
\nc{\egm}[1]{e^{-}_{#1}}
\nc{\egpm}[1]{e^{\pm}_{#1}}
\nc{\egmp}[1]{e^{\mp}_{#1}}

\nc{\Ce}{\cC}

\nc{\xpm}[1]{x^{\pm}_{#1}}
\nc{\xmp}[1]{x^{\mp}_{#1}}
\nc{\xp}[1]{x^{+}_{#1}}
\nc{\xm}[1]{x^{-}_{#1}}
\nc{\xz}[1]{\xi_{#1}}

\nc{\phipm}[1]{\phi^{\pm}_{#1}}
\nc{\phip}[1]{\phi^{+}_{#1}}
\nc{\phim}[1]{\phi^{-}_{#1}}

\nc{\chev}{\omega}

\nc{\weightspace}[2]{{{#1}_{#2}}}
\nc{\wsp}[2]{\weightspace{#1}{#2}}
\nc{\wts}[1]{\mathsf{wt}(#1)}
\nc{\hwL}[1]{L(#1)}
\nc{\catO}[3]{\O_{#1}^{#2}{#3}}

\nc{\qstr}[2]{\Sigma_{#1,#2}}

\nc{\evrep}[2]{V_{#1}(#2)}
\nc{\shrep}[2]{{#1}(#2)}
\nc{\Lshrep}[2]{\Lfml{#1}{#2}}
\nc{\Pshrep}[2]{\Pfml{#1}{#2}}

\nc{\qstrep}[2]{V(\qstr{#1}{#2})}
\nc{\qsrep}[1]{V(#1)}

\nc{\HL}[1]{\mathcal{C}_{#1}}

\nc{\Oint}{{\cO^{\sf int}}}

\nc{\shift}[1]{\Sigma_{#1}}
\nc{\tshift}[1]{\Sigma^{\tau}_{#1}}
\nc{\shiftm}[2]{{#1}_{#2}}

\nc{\shifta}[1]{{\chi}_{#1}}

\nc{\Deltaop}{\Delta^{\sf op}}

\nc{\fml}[2]{{#1}[\negthinspace[#2]\negthinspace]} 
\nc{\Lfml}[2]{{#1}(\negthinspace(#2)\negthinspace)} 
\nc{\Pfml}[2]{{#1}\{#2\}} 

\nc{\qsl}[1]{U_q({\mathfrak{sl}}_{#1})}
\nc{\qasl}[1]{U_q(\wh{\mathfrak{sl}}_{#1})}
\nc{\qlsl}[1]{U_q(L\mathfrak{sl}_{#1})}

\nc{\brS}[1]{S_{#1}} 
\nc{\brSg}[1]{\widetilde{s}_{#1}} 

\nc{\Br}[1]{\mathscr{B}_{#1}} 
\nc{\RBr}[1]{\mathscr{RB}_{#1}} 

\nc{\qWS}[1]{S_{#1}}
\nc{\LT}[1]{T_{#1}}

\nc{\tcorr}[1]{\ka_{#1}} 
\nc{\bt}[1]{\mathcal{S}_{#1}} 

\nc{\adt}[1]{\mathcal{T}_{#1}} 
\nc{\adbt}[1]{\mathcal{T}_{#1}} 

\nc{\Rcorr}[1]{\eta_{#1}}

\nc{\intg}{\mathbf{W}^{\mathsf{int}}}
\nc{\Vect}[1]{\operatorname{Vect}_{#1}}

\nc{\corank}{\operatorname{cork}}

\nc{\gsat}[1]{\mathsf{GSat}(#1)} 
\nc{\auxgsat}[2]{\mathsf{Sat}(#1;#2)} 
\nc{\sat}[1]{{\mathbf{S}}} 
\nc{\tsat}{\theta} 
\nc{\zsat}{\zeta} 
\nc{\tsatq}{\tsat_{q}} 
\nc{\zsatq}{\zsat_{q}} 

\nc{\tinv}[1]{\theta_{#1}}

\nc{\oi}{\mathsf{oi}} 

\nc{\Parsetc}{\bm{\Gamma}} 
\nc{\Parsets}{\bm{\Sigma}}

\nc{\Parc}{\bm{\gamma}} 
\nc{\Pars}{\bm{\sigma}}

\nc{\parc}[1]{\bm{\gamma}_{#1}} 
\nc{\pars}[1]{\bm{\sigma}_{#1}}

\nc{\ctheta}{\theta} 

\nc{\qtheta}{\theta_q} 
\nc{\qthetat}{\widetilde{\theta}_q} 

\nc{\QSP}{U_q(\mathfrak{k})}
\nc{\Uqk}{U_q(\mathfrak{k})}
\nc{\Uqkp}{U_q({\mathfrak{k}'})}

\nc{\wt}{\widetilde}

\nc{\Ieq}{I_{\sf eq}} 
\nc{\Idiff}{I_{\sf diff}} 
\nc{\Ins}{I_{\sf ns}} 
\nc{\aIeq}{\aIS_{\sf eq}} 
\nc{\aIdiff}{\aIS_{\sf diff}} 
\nc{\aIns}{\aIS_{\sf ns}} 

\nc{\bg}[1]{b_{#1}} 
\nc{\Bg}[1]{B_{#1}} 

\nc{\QK}[1]{\Ups_{#1}} 
\nc{\RM}[1]{{R}_{#1}} 
\nc{\QR}[1]{{\Xi}_{#1}} 
\nc{\sRM}[2]{{{R}}_{#1}(#2)} 
\nc{\sRMv}[2]{{\widecheck{R}}_{#1}(#2)} 
\nc{\rRM}[2]{{\mathbf{R}}_{#1}(#2)} 
\nc{\rRMC}[3]{{\mathbf{R}}^{#1}_{#2}(#3)} 
\nc{\rRMv}[2]{\widecheck{\mathbf{R}}_{#1}(#2)} 
\nc{\rRMt}[2]{\widetilde{\mathbf{R}}_{#1}(#2)} 
\nc{\KM}[1]{K_{#1}} 
\nc{\sKM}[2]{K_{#1}(#2)} 
\nc{\rKM}[2]{\mathbf{K}_{#1}(#2)} 
\nc{\trKM}[2]{\wt{\mathbf{K}}_{#1}(#2)} 
\nc{\TKM}[1]{{\bbK}_{#1}} 

\nc{\zeroKM}[2]{K^0_{#1}(#2)} 
\nc{\inftyKM}[2]{K^{\infty}_{#1}(#2)} 

\nc{\tsRM}[2]{{{R}}^{\tau}_{#1}(#2)} 
\nc{\tsRMv}[2]{{\widecheck{R}}^{\tau}_{#1}(#2)} 

\nc{\auxsat}[1]{\bfT} 

\nc{\hext}[2]{{#1}[\negthinspace[#2]\negthinspace]} 

\nc{\binomb}[2]{\genfrac{[}{]}{0pt}{}{#1}{#2}}

\nc{\GQSP}{\cG_{\tsat,\parc{}}}

\nc{\gpsi}{\ol\psi}
\nc{\gJ}{\ol J}
\nc{\gK}{\ol K}
\nc{\gTKM}[1]{\ol{\mathbb{K}}_{#1}}
\nc{\TM}{\mathscr{T}}

\nc{\Supp}{\mathsf{Supp}}
\nc{\twistrep}[2]{{#1}^{#2}}  

\nc{\wtbeta}{\wt{\beta}}

\nc{\Gg}{\cG} 
\nc{\gau}{{\bf g}} 

\nc{\op}{\operatorname{op}}
\nc{\cop}{\operatorname{cop}}

\nc{\sclQK}{\mathbf{Y}} 

\nc{\tproj}[1]{[#1]_\tsat}

\nc{\hgt}{{\mathsf{ht}}}

\nc{\Vpsi}{V^{\psi}}
\nc{\Wpsi}{W^{\psi}}
\nc{\VWpsi}{(V\ten W)^{\psi}}

\nc{\sint}{{\operatorname{int}}}
\nc{\pint}[1]{{{#1}\negthinspace\operatorname{-int}}}
\nc{\Xint}{{{{X}\negthinspace\operatorname{-int}}}}
\nc{\WUqg}{\cW} 
\nc{\OUqg}{\cO} 
\nc{\OintUqg}{\cO_{\sint}} 
\nc{\POUqg}[1]{\cO_{\pint{#1}}} 
\nc{\vect}{\operatorname{Vect}}

\nc{\FF}[2]{\sff_{#1}^{#2}}

\nc{\COUqg}[2]{\End(\FF{#1}{#2})}
\nc{\COintUqg}[1]{\End(\FF{\sint}{#1})}
\nc{\CPOUqg}[2]{\End(\FF{\pint{#1}}{#2})} 
\nc{\CWUqg}[1]{\End(\FF{\cW}{#1})}
\nc{\WO}[1]{\operatorname{C}^{#1}}
\nc{\WUqk}{\cW_{\theta}}
\nc{\CWUqk}[1]{\End(\sff_\theta^{#1})}
\nc{\Dr}{\cD}
\nc{\Drt}{\cE_{\theta}}
\nc{\DrX}{\cE^{\op}_X}
\nc{\TQK}[1]{\Xi_{#1}}

\nc{\PROPB}{\mathsf{PROP/B}}
\nc{\PROB}{\mathsf{PROB}}
\nc{\PROP}{\mathsf{PROP}}
\nc{\bfb}{\mathbf{b}}
\nc{\G}{\cG}
\nc{\Fun}{\operatorname{Fun}}

\nc{\monactp}{\vartriangleleft}
\nc{\monactm}{\blacktriangleleft}

\rnc{\UqLg}{U_q(\mathfrak{L})} 
\rnc{\gfin}{\mathring{\g}}

\nc{\fdUqL}{\cC}
\nc{\fdUqk}{\cC_{\theta}}

\nc{\rnRM}[2]{\mathbf{R}^{\operatorname{norm}}_{#1}(#2)}
\nc{\rTKM}[2]{\mathbb{K}^{\operatorname{trig}}_{#1}(#2)}
\nc{\rnTKM}[2]{\mathbb{K}^{\operatorname{norm}}_{#1}(#2)}
\nc{\rnKM}[2]{\mathbf{K}^{\operatorname{norm}}_{#1}(#2)}

\nc{\bmb}{{\bm \beta}}

\nc{\Ct}{\mathbf{t}}
\nc{\Cb}{\mathbf{b}}
\nc{\FSeq}{\operatorname{Seq}}
\nc{\BSeq}{\operatorname{BSeq}}

\AtBeginDocument{%
   \def\MR#1{}
}

\title[Boundary transfer matrices from quantum symmetric pairs]{Boundary transfer matrices arising from quantum symmetric pairs}

\author[A. Appel]{Andrea Appel} 
\address{Dipartimento di Scienze Matematiche, 
	Fisiche e Informatiche, Universit\`a di Parma, 
	INdAM - GNSAGA,
	and INFN - Gruppo Collegato di Parma, 
	Parco Area delle Scienze 53/A, 
	43124 Parma (PR), Italy}
\email{\href{mailto:andrea.appel@unipr.it}{andrea.appel@unipr.it}}

\author[B. Vlaar]{Bart Vlaar}
\address{
Beijing Institute of Mathematical Sciences and Applications, No.\ 544, Hefangkou Village, Huaibei Town, Huairou, Beijing, China}
\email{\href{b.vlaar@bimsa.cn}{b.vlaar@bimsa.cn}}

\dedicatory{To Tom Koornwinder, on the occasion of his 80th birthday}
\thanks{The first author is partially supported by 
an INdAM Project Grant 2024 and the PRIN grant 2022HMTBLL.
Both authors are partially supported by the International Scientists Program of the Beijing Natural Sciences Foundation (grant number {IS24003}).}



\nc{\coloneqq}{=} 

\begin{document}
	

\begin{abstract}
We introduce a universal framework for boundary transfer matrices, inspired by Sklyanin’s two-row transfer matrix approach for quantum integrable systems with boundary conditions.
The main examples arise from quantum symmetric pairs of finite and affine type. 
As a special case we recover a construction by Kolb in finite type. 
We review recent work on universal solutions to the reflection equation and highlight several open problems in this field.
\end{abstract}


\maketitle

\setcounter{tocdepth}{1}
\tableofcontents


\section{Introduction}

\subsection{Quantum symmetric pairs}

Motivated by questions in harmonic analysis, quantum deformations of the homogeneous spaces $G/G^\Theta$ were extensively studied in the 1990s.
Here $G$ is a connected complex Lie group, while $G^\Theta$ is the subgroup of fixed points with respect to an involutive automorphism $\Theta\colon G \to G$. 
Let $\theta= {\sf d} \Theta$ be the involutive automorphism of the corresponding Lie algebra $\fkg$ and consider the fixed-point subalgebra $\fkg^\theta$.

By \cite{Jim85,Dri86}, the universal enveloping algebra $U(\fkg)$ admits a Hopf algebra 
deformation $U_q(\fkg)$, called a quantized universal enveloping algebra or a Drinfeld-Jimbo quantum group. This  is a non-commutative non-cocommutative Hopf algebra, which reduces to  $U(\fkg)$ as $q \to 1$.

Constructing a suitable deformation of $U(\fkg^\theta)$ is somewhat more subtle.
Given the scarcity of Hopf subalgebras of $U_q(\fkg)$, a natural alternative is provided by the following basic observation. 
Let $C_q \subseteq U_q(\fkg)$ be a two-sided coideal subalgebra, \ie
\[
\Delta(C_q) \subseteq U_q(\fkg) \otimes C_q + C_q \otimes U_q(\fkg).
\]
with a well-defined limit $C\subseteq U(\g)$ as $q \to 1$.
Relying on a geometric series formula of the antipode $S: U_q(\fkg) \to U_q(\fkg)$, one observes that the subbialgebra $C$ is a Hopf subalgebra of $U(\fkg)$ and hence it has the form $C = U(\fkk)$ for some Lie subalgebra $\fkk \subseteq \fkg$.

T. Koornwinder found a two-sided coideal subalgebra of $U_q(\fksl_2)$ as a deformation of $U(\fkso_2) \subset U(\fksl_2)$ in \cite{Koo93}. More generally, deformations of $U(\fkg^\theta)$ were obtained as one-sided coideal subalgebras of $U_q(\fkg)$ in \cite{NS95,Nou96,DN98,GIK98}. Similar constructions were already considered in \cite{GK91}, although the coideal property was not emphasized.

In parallel, G. Letzter introduced one-sided coideal deformations of $U(\fkg^\theta)$ for any symmetric pair $(\fkg,\fkg^\theta)$, see \cite{Let97,Let99,Let02}. Their definition is given directly in terms of the generators of $U_q(\fkg)$.
This construction was extended by S. Kolb to the case where $\fkg$ is any symmetrizable Kac-Moody algebra in \cite{Kol14}, recovering as a special case the twisted quantum loop algebras, also known as twisted q-Yangians, see \cite{MRS03, CGM14}. 
Quantum deformations $\Uqk \subset \Uqg$ of a more general class of subalgebras $\fkk \subseteq \fkg$ which are not necessarily fixed-point subalgebras, which we refer to as \emph{pseudo-fixed-point subalgebras}, were obtained in \cite{RV20,RV22}.
We refer to all such pairs $(\Uqg,\Uqk)$ as quantum (pseudo-)symmetric pairs (QSP), and to $\Uqk$ as a QSP subalgebra.

\subsection{Universal K-matrices}

In the works \cite{NS95,Nou96,MRS03,KS09,CGM14} the QSP subalgebras were defined in terms of a K-matrix, \ie~a matrix solution to a reflection equation, using an approach akin to the R-matrix presentation for quantum groups outlined in \cite{FRT90}.
Conversely, in the Letzter-Kolb formalism the QSP subalgebra $U_q(\fkk) \subseteq U_q(\fkg)$ naturally gives rise to a universal solution of a (generalized) reflection equation.
In \cite{BW18a}, 
H. Bao and W. Wang introduced  
a twisted $\Uqk$-intertwiner, called \emph{quasi K-matrix}, and proved its existence for a family of QSPs extending $U_q(\fksl_N)$. 
This construction has been generalized across the range of QSPs, see \cite{BW18b,BK19,RV20,AV22,WZh23}.\\

In \cite{AV22}, we introduced the notion of a cylindrical structure on $(A,B)$ where $A$ is a quasitriangular bialgebra and $B \subseteq A$ is a right coideal subalgebra. 
It consists of a suitable algebra automorphism $\psi: A \to A$ and a twisted 
$B$-intertwiner $K \in A^\times$, called a \emph{basic K-matrix},
satisfying a twisted reflection equation 
with respect to the universal R-matrix $R$ of $A$. 
Cylindrical structures can be \emph{gauge-transformed} by elements in $A^\times$. 
This gauge action can be used to bring either $K$ or $\psi$ to a more convenient form.

This approach has been instrumental in \cite{AV25a}, leading to a uniform and systematic construction of spectral K-matrices (\ie parameter-dependent matrix solutions of the reflection equation) on finite-dimensional representations of quantum affine algebras.
The key idea is that the twist automorphism $\psi$ can be used to accomplish the inversion of the spectral parameter. 
Specifically, when $\g$ is an (untwisted) affine Lie algebra, the universal K-matrix gives rise to a formal Laurent series on any finite-dimensional module over the quantum affine algebra $U_q(\fkg') \subset U_q(\fkg)$. 
It satisfies a generalized reflection equation of the type considered in \cite{FM91,KS92,Che92}, which includes the original "untwisted" reflection equation from \cite{Che84,Skl88}. 
Moreover, on irreducible modules, it yields rational solutions of the same generalized reflection equation, called a \emph{trigonometric K-matrix}.\\

By the coideal property of $B$, $\Mod(B)$ is naturally equipped with a module structure over the monoidal category $\Mod(A)$. 
Since $\Mod(A)$ is braided, it is natural to require that the action on $\Mod(B)$ is compatible with the braiding induced by the R-matrix. 
This is equivalent to the existence of an element $\TKM{} \in B \ot A$ which satisfies similar relations as $K$, and it is therefore referred to as a \emph{tensor K-matrix}. 
The universal framework encoding the action of $\TKM{}$ has been developed in \cite{DKM03,Enr07,Bro13,BZBJ18}, with further developments and applications in \cite{DCNTY19,Kol20,KY20,LBG23,Lem23,Xu23,AV25b}.
Roughly, the tensor K-matrix yields a {\em braiding} on $\Mod(B)$, and induces an action of the cylindrical braid group.

In presence of a cylindrical structure, the element
\eq{ 
\TKM{} \coloneqq  (R^\psi)_{21} \cdot (1 \ot K) \cdot R, \qq R^\psi \coloneqq  (\psi \ot \id)(R),
} 
is almost automatically a tensor K-matrix, except for the \emph{support condition} $\TKM{}\in B\ten A$. 
In the case of an arbitrary QSP, we proved in \cite{AV25b} that, for every cylindrical structure constructed in \cite{AV22}, the operator $\TKM{}$ is indeed a tensor K-matrix.\footnote{Motivated by this result, in this paper we include the support condition in the definition of a cylindrical structure, in contrast with \cite{AV22, AV25a, AV25b}.}

\subsection{Present work} 
In this paper, we build on the results from \cite{AV22, AV25b} to provide a uniform construction of a \emph{universal boundary transfer matrix} and references therein, in the form of a ring homomorphism 
\eq{
	\TM: [\Mod_{\sf fd}(A)] \to B, \qq [V] \mapsto \TM_V = \Tr_2(\TKM{})
}
where $[\Mod_{\sf fd}(A)]$ denotes the Grothendieck ring of finite-dimensional $A$-modules. 
Analogous to the case of transfer matrices arising from a universal R-matrix, we define $\TM$ by letting the second leg of $\TKM{} \in B \ot A$, to act on a finite-dimensional $A$-module $V$ and then taking the trace over $V$.
Note also that in \cite{Lem23} universal transfer matrices of this type are considered for the q-Onsager algebra.\\ 

This construction, however, does not quite work if we consider an arbitrary cylindrical structure: it does not behave well with respect to tensor products in the second leg. 
For a large class of bialgebras $A$, this can be resolved by gauge-transforming, resulting in a dinstinguished cylindrical structure.
We call the relevant gauge transformation, which appears as a factor in the formula for the transformed $\TM$, a \emph{boundary} gauge transformation. 
This modification is inspired by the theory of (non-universal) 2-boundary transfer matrices due to E.~Sklyanin, see \cite{Skl88} and also \cite{MN91,FSHY97,Vla15}.

\Omit{
\andreacomment{I am kind of stunned by the observation above. It is not really correct to say that $\TKM{}$ does not behave well, because eventually its gauge transformed $\gTKM{}$ is still (just) a tensor K-matrix! However, the twist associated to $\gTKM{}$ satisfies $\gJ=\Ad(1\ten D)(R^{\gpsi})$, therefore
\[
1\ten\Delta(\gTKM{})=(1\ten 1\ten D)(R^{\gpsi})_{23}^{-1}(1\ten 1\ten D^{-1})\gTKM{13}R^{\gpsi}_{23}
\gTKM{12}\]
Why is this identity better if not because we prove it to be so? This shows that the problem is not really the tensor K-matrix, but the Drinfeld twist associated to the K-matrix. It is in {\bf this} sense that $\TKM{}$ does not behave well, because it involves the {\bf wrong} Drinfeld twist. By gauging $\TKM{}$ we still get a tensor K-matrix, but with a very different Drinfeld twist. This is truly nice, because so far we have been appreciating the usefulness of the twist operator, while this construction really highlight the role of the Drinfeld twist. If not yet clear, this is a stunning proof of the validity of our \emph{full force} definition of cylindrical structures.}

\bartcomment{Indeed, very nice. It seemed to me too technical to talk about Drinfeld twists in the introduction. But, by your point, maybe we should. 
I rewrote the above paragraph a little without going as far as that. It ties in with my comments at the start of Section 2.5 and somewhat also with referee comment 6. Some old comments below. \\

Now it seems to me more natural to set it up as follows (rather than directly imposing a condition on the gauge transformation).
Call a twist pair $(\psi,J)$ \emph{well-behaved} (replace by nicer word!) if for all $W \in \Mod_{\sf fd}(A)$ we have
\[ 
J_{\bullet,W} = ((((R^\psi)_{\bullet,W}^{\sf t_2})^{-1})^{\sf t_2})^{-1}.
\]
For a balanced Hopf algebra, this amounts to $J_{\bullet,W} = \Ad(1 \ot D_W)(R^\psi_{\bullet,W})$. 
A natural sufficient condition for this is
\[
J= \Ad(1 \ot D)(R^\psi).
\] 
It is equivalent if $\Mod_{\sf fd}(A)$ separates points. For well-behaved twist pairs we have Theorem \ref{thm:t:Grothendieck}.\\

It is natural to ask if a given cylindrical structure can be gauge-transformed to a well-behaved one. 
In Chapter 3, we essentially describe this. 
Let $(\psi_0,J_0)$ be a given twist pair. 
Suppose we have a solution of $\Delta(K') = J_0^{-1} \cdot (1 \ot K') \cdot R^{\psi_0} \cdot (K' \ot 1)$.
Also let $D$ be a sovereign element.
Then we can gauge transform $(\psi_0,J_0)$ to a well-behaved twist pair by a general invertible element of $A$, written in the form
\[
g = D \cdot a \cdot (K')^{-1}
\]
with $a \in A^\times$ arbitrary, if and only if
\[
\TKM{}' \cdot \Delta(a^{-1}) \cdot (1 \ot a) = (a^{-1} \ot 1) \cdot \TKM{}'.
\]

It implies $\eps(a)=1$.
The special desired solution $a = 1$ stands out.
Can we develop it further towards some classification result up to central grouplike elements?
}
}

When $A$ is a balanced Hopf algebra, a remarkable example of a boundary gauge transformation is given by the \emph{generalized dual K-matrices}. 
These are elements of the form $\wt K = D \cdot (K')^{-1}$, where $D$ is the sovereign element\footnote{That is, $D=b^{-1}u$, where $b$ is the balance and $u$ is the Drinfeld element. Note that $D$ is group-like and $\Ad(D)=S^2$, see Section~\ref{ss:balance-sovereign}.}, and $K'$ is a basic K-matrix, suitably compatible with $K$.
In particular, $\wt K$ is itself a solution of a generalized reflection equation.
For any such $\wt K$, we obtain a ring homomorphism $\TM: [\Mod_{\sf fd}(A)] \to B$
by setting $\TM^{(V)}=\Tr_2(\gTKM{\bullet, V})$
where $\gTKM{}=(1\ten\wt K)\cdot \TKM{}$ (see Theorem.~\ref{thm:t:Grothendieck}).\\

The construction has interesting applications to the theory of quantum symmetric pairs of finite and affine types.
\begin{enumerate}\itemsep0.25cm
\item Let $(\Uqg,\Uqk)$ be a QSP of finite type equipped with the cylindrical structure from \cite{BK19, Kol20}. 
When $K'$ is the balance, so $\wt K$ is the Drinfeld element, the map $\TM$ specializes to the ring homomorphism $\bm\Phi$ constructed by Kolb in \cite[Sec.~4]{Kol20}, and it is therefore surjective over $Z(\Uqk)$ (see Proposition~\ref{prop:kolb}).

We observe, however, that in certain cases $\TM$ reduces to a rather trivial map. 
For instance, this occurs whenever $\wt K$ and $\TKM{}$ are defined in terms of the same basic K-matrix $K$ (see Proposition \ref{lem:transfermatrix:trivial}).

\item
Let $(\Uqg,\Uqk)$ be a QSP of affine type. Then, the procedure outlined above yields a ring homomorphism 
\eq{
	\TM(z): [\Mod_{\sf fd}(\Uqg)] \to \Lfml{\Uqk}{z}
}
(see Theorem~\ref{thm:spectraltransfermatrix}).
It provides a universal analogue of the construction of the transfer matrix given by Sklyanin in \cite{Skl88}, extending it to more general reflection equations, on a par with those considered  by I. Cherednik in \cite[Sec.~5]{Che92} in the relation with the two-boundary quantum KZ equations.
Remarkably, if $\wt K$ and $\TKM{}$ are defined in terms of the same basic K-matrix $K$, we still obtain highly nontrivial elements in $\Lfml{\Uqk}{z}$ in stark contrast to the finite type case.
It is natural to interpret the image of $\TM(z)$ in $\Lfml{\Uqk}{z}$ as QSP analogue of a Bethe subalgebra, in the spirit of \cite{BK05,BB13,Lem23}. 
	
The map $\TM(z)$ can be regarded as a first step in the definition of a boundary analogue of the q-characters for quantum affine algebras developed by E. Frenkel, E. Mukhin and N. Reshetikhin \cite{FR99,FM01}. 
We outline further potential applications in the representation theory of $\Uqg$ and $\Uqk$ in Section \ref{sec:reptheoryapplications}.
\end{enumerate}


\subsection{Notation}

We will introduce some standard notations which will be used throughout this paper for any unital associative $\bsF$-algebra $A$ for any field $\bsF$.
We denote conjugation by $a \in A^\times$ by $\Ad(a)$.

For any $\bsF$-linear space $V$, consider the algebraic dual $V^\vee \coloneqq  \Hom_\bsF(V,\bsF)$ and the \emph{(right) partial transpose}, which is the linear map 
\eq{
	{}^{\sf t_2}: A \ot \End_{\bsF}(V) \to A \ot \End_{\bsF}(V^\vee) 
}
uniquely defined by 
$(a \ot f)^{\sf t_2} = a \ot f^{\sf t}$ for $a \in A$ and $f \in \End_{\bsF}(V)$.
Further, if $\dim(V)<\infty$, the \emph{(right) partial trace} is the linear map
\eq{
	\Tr_2: A \ot \End_{\bsF}(V) \to A
}
uniquely defined by $\Tr_2 (a \ot f) = \Tr(f) a$ for $a \in A$ and $f \in \End_{\bsF}(V)$.
Note the following basic identities:
\eq{
\label{linalg} \Tr_2 (\Phi \cdot \Psi) = \Tr_2 (\Phi^{\t_2} \cdot \Psi^{\t_2}), \qq (\Phi \cdot (1_A \ot f))^{\t_2} = (1_A \ot f^\t) \cdot \Phi^{\t_2},
}
for all $\Phi,\Psi \in A \ot \End_\bsF(V)$ and $f \in \End_\bsF(V)$.

If $V$ is an $A$-module, we will write $a_V$ for the action of $a \in A$ on $V$.
More generally, let $T \in A^{\ot n}$ for $n \in \Z_{\ge 1}$ and let $V_1,\ldots,V_n$ be $A$-modules.
Then we will write $T_{V_1,\ldots,V_n}$ for the action of $T$ on $V_1 \ot \cdots \ot V_n$. 
Moreover, for any $1 \le i \le n$, 
\[
T_{V_1,\ldots,V_{i-1},\bullet,V_{i+1},\ldots,V_n} \in \End_{\bsF}(V_1 \ot \cdots \ot V_{i-1}) \ot A \ot \End_{\bsF}(V_{i+1} \ot \cdots \ot V_n)
\]
is obtained from $T$ by letting each tensor factor of $T$ except the $i$-th act on the corresponding $A$-module.

Let $\psi : A \to A$ be an algebra automorphism.
For any $T \in A \ot A$, we will write
\eq{ 
T^\psi \coloneqq  (\psi \ot \id)(T), \qq T^{\psi\psi} = (\psi \ot \psi)(T).
}
Further, we will write $V^\psi$ for the correspondingly $\psi$-twisted module structure on the same underlying $\bsF$-linear space. 
In other words, for all $a \in A$ we have $a_{V^\psi}  = \psi(a)_V$.

\subsection{Acknowledgments}
The authors would like to thank Pascal Baseilhac, Azat Gainutdinov, David Hernandez and Stefan Kolb for stimulating discussions. 
Some of the new results contained in this manuscript were announced at the conference Recent Advances in Quantum Integrable Systems 2024 at Laboratoire d'Annecy de Physique Théorique. B.V. thanks the organizers for the kind invitation and hospitality.


\section{Boundary transfer matrices for cylindrical bialgebras} \label{sec:cylstrbalHopf}

A cylindrical structure is an extra structure placed on a quasitriangular bialgebra $A$ and a right coideal subalgebra $B$, specifically designed to encode a universal solution to a generalized reflection equation \cite{AV22,AV25b}. In this section, we determine a sufficient condition for a cylindrical structures on $(A, B)$ to yield a ring homomorphism from the Grothendieck ring of $\Mod_{\sf fd}(A)$ to a commutative subring of $B$. This is a \emph{boundary} analogue of the transfer matrix map considered in \cite{FR99}, providing a universal and generalized version of two-row transfer matrices, originally due to E. Sklyanin \cite{Skl88}.

\subsection{Quasitriangular bialgebras and twist pairs}

Let $A$ be a bialgebra over an arbitrary field $\bbF$, with coproduct $\Delta: A \to A \ot A$ and opposite coproduct $\Delta^{\sf op} \coloneqq {\sf flip} \circ \Delta$. 
We assume that $A$ is \emph{quasitriangular} \cite{Dri86}, \ie there exists $R \in (A \ot A)^\times$ such that
\begin{gather}
	\label{R:intw} \Ad(R) \circ \Delta  = \Delta^{\sf op}, \\
	\label{R:coproduct} (\Delta \ot \id)(R) = R_{13} \cdot R_{23}, \qq (\id \ot \Delta)(R) = R_{13} \cdot R_{12}.
\end{gather}
Note that $(A,\Delta^{\sf op}, R_{21})$ is also a quasitriangular bialgebra, denoted $A^{\sf cop}$.
New quasitriangular bialgebras can be obtained from $A$ in essentially two ways.
\begin{itemize}\itemsep0.25cm
\item
If $\psi: A \to A$ is an algebra automorphism, then we denote by $A^\psi$ the quasitriangular bialgebra with the underlying algebra structure unchanged and the coproduct and the R-matrix are given by
	\eq{
		\Delta^\psi \coloneqq (\psi \ot \psi) \circ \Del \circ \psi^{-1}
		\qq\mbox{and}\qq
		R^{\psi \psi} \coloneqq (\psi \ot \psi)(R).
	}
We refer to $\psi$ as a \emph{twist automorphism}.
\item
If $J \in (A \ot A)^\times$ is a \emph{Drinfeld twist}, \ie 
	\eq{ \label{J:cocycle}
		J_{12} \cdot (\Delta \ot \id)(J) = J_{23} \cdot (\id \ot \Delta)(J)
	}
and $(\eps \ot \id)(J) = 1 = (\id \ot \eps)(J)$, where $\eps$ is the counit, then we denote by $A_J$ the quasitriangular bialgebra with the underlying algebra structure unchanged and the coproduct and the R-matrix given by
	\eq{
		\Delta_J = \Ad(J) \circ \Del
		\qq\mbox{and}\qq
		R_J = J_{21} \cdot R \cdot J^{-1}.
	}
\end{itemize}
A \emph{twist pair} $(\psi,J)$ is the datum of a twist automorphism $\psi$ and a Drinfeld twist $J$ such that $(A^{\sf cop})^\psi = A_J$, \ie
\eq{
(\psi \ot \psi) \circ \Delta^{\sf op} \circ \psi^{-1} = \Ad(J) \circ \Del
\qq\mbox{and}\qq
 R^{\psi \psi}_{21} = J_{21} \cdot R \cdot J^{-1}
}
(and hence $\eps \circ \psi^{-1} = \eps$). Equivalently, the Drinfeld twist $J$ measures the failure of $\psi$ to be an isomorphism of quasitriangular bialgebras\footnote{Indeed, in the terminology introduced in \cite{Dav07}, a twist pair is the same as a {\em twisted} morphism of quasitriangular bialgebras $A^{\sf cop} \to A$.} $A^{\sf cop} \to A$.

\begin{example}
If $\om: A^{\sf cop} \to A$ is an isomorphism of quasitriangular bialgebras, then $(\om, 1 \ot 1)$ is a twist pair. Instead, if $\upsilon: A \to A$ is an automorphism of quasitriangular bialgebras, then $(\upsilon, R)$ and $(\upsilon,R_{21}^{-1})$ are twist pairs. \rmkend
\end{example}

\subsection{Cylindrical structures \cite{AV22}}


Let $B \subseteq A$ be a right coideal subalgebra, \ie $\Delta(B) \subseteq B \ot A$.

\begin{definition}
A \emph{cylindrical structure} on $(A,B)$ is a datum $(\psi, J, K)$ where $(\psi,J)$ is a twist pair for $A$ and $K \in A^\times$ is a \emph{basic K-matrix}, \ie it satisfies\footnote{Note that $\psi$ is not assumed to preserve $B\subseteq A$. This is consistent with the case of the intertwining equation for the R-matrix \eqref{R:intw}, where ${\sf flip}$ does not necessarily preserve $\Delta(A) \subset A \ot A$.}  
\begin{align}
\label{K:intw} \Ad(K)|_B &= \psi|_B, \\
\label{K:support} (R^\psi)_{21} \cdot K_2 \cdot R &\in B \ot A, \\
\label{K:coproduct} \Delta(K) &= J^{-1} \cdot K_2 \cdot R^\psi \cdot K_1. 
\end{align}
\end{definition}
\noindent
As a consequence, the \emph{generalized reflection equation} is satisfied in $A \ot A$:
\eq{ \label{K:RE}
	K_1 \cdot (R^\psi)_{21} \cdot K_2 \cdot R = (R^{\psi \psi})_{21} \cdot K_2 \cdot R^\psi \cdot K_1.
} 

\begin{remark}
This notion\footnote{
	The support condition \eqref{K:support} is not included in the definition given in \cite{AV22}.
	Instead, it was considered in \cite{Kol20}, in the case $\psi$ is an automorphism of quasitriangular bialgebras, and, more generally, in \cite{AV25b}.
}  of a cylindrical structure, given in \cite{AV22}, naturally generalizes the formalism introduced by T. tom Dieck and R. H\"aring-Oldenburg in the case $\psi=\id$ \cite{tDHO98}
and by M. Balagovi\'c and S. Kolb in the case $\psi\colon A\to A$ is an automorphism of quasitriangular bialgebras and $J=R_{21}^{-1}$ \cite{BK19}. \rmkend
\end{remark}

\begin{example}
\hfill
\begin{enumerate}\itemsep1mm
\item Let $g\in A^\times$ be a group-like element. Then, $(\psi,J,K) = (\Ad(g),R,g)$ is a cylindrical structure on $(A,A)$.
\item Quasitriangular structures can be regarded as cylindrical structures as follows. 
Let $A^{\sf cop} \ot A$ be the quasitriangular bialgebra with coproduct $(\id \ot {\sf flip} \ot \id) \circ (\Del^{\sf op} \ot \Del)$ and universal R-matrix $R_{31} \cdot R_{24}$, and consider $\Delta(A)\subseteq A^{\sf cop}\ten A$ as a right coideal subalgebra. Then, $(\psi, J, K) = ({\sf flip}, 1 \ot 1, R)$ is a cylindrical structure on $(A^{\sf cop} \ot A,\Delta(A))$. 
\rmkend
\end{enumerate}
\end{example}

\subsection{Tensor K-matrices} \label{sec:tensorK}
Since $B$ is a right coideal subalgebra of $A$, $\Mod(B)$ is naturally a right module category over the braided monoidal category $\Mod(A)$. 
If the support condition \eqref{K:support} holds, then $\Mod(B)$ is equipped with an extra structure given in terms of the operator
\eq{ \label{tensorK:def}
	\TKM{} \coloneqq (R^\psi)_{21} \cdot K_2 \cdot R \in B \ot A,
} 
see \cite[App.~B]{AV25b} and \cite{LBG23,Lem23} for the more general setting of right comodule algebras. $\TKM{}$ is referred to as a \emph{tensor K-matrix} and satisfies 
\begin{align}
\Ad(\TKM{})|_{\Delta(B)} &= (\id \ot \psi)|_{\Delta(B)}, \label{tensorK:intw} \\
(\Delta \ot \id)(\TKM{}) &= (R^\psi)_{32} \cdot \TKM{13} \cdot R_{23}, \label{tensorK:coproduct1} \\
(\id \ot \Delta)(\TKM{}) &= J_{23}^{-1} \cdot \TKM{13} \cdot (R^\psi)_{23} \cdot \TKM{12}. \label{tensorK:coproduct2} 
\end{align}
Note that $K$ can be recovered from $\TKM{}$ by $K = (\eps \ot \id)(\TKM{})$.

\begin{remark}\hfill
\begin{enumerate}\itemsep1mm
\item As an operator on modules, the action of $\TKM{}$ gives a {\em braiding} on the module category $\Mod(B)$ over $\Mod(A)$, and yields an action of the {cylindrical braid group}, see \cite{Kol20}. 
However, in the setting of quantum symmetric pairs, $\Mod(A)$ is replaced by a proper subcategory $\cC\subset\Mod(A)$, which may not be closed under twisting by $\psi$. 
In these cases, the framework of braided module categories does not apply.
To remedy this, we introduced in \cite[App.~C]{AV25b} the notion of a {\em boundary} bimodule category.
\item
The case $\psi = \id$ was already considered in \cite{Enr07,Bro13,Kol20}.
In \cite[App.~B.5]{AV25b}, we also highlight the similarities with the framework of weakly universal K-matrices from \cite{KY20}.
\rmkend
\end{enumerate}
\end{remark}


\subsection{Gauge transformations and invariants}\label{ss:gauge}

There is a natural action of the group $A^\times$ on the set of cylindrical structures on $(A,B)$, which we refer to as \emph{gauge transformation}, see 
\cite[Rmk.~8.11]{AV22} and \cite[Sec.~3.5]{AV25a}. 
The action of $g \in A^\times$ is given by
\begin{equation}
	\label{cylindrical:gauge} 
	\begin{tikzcd}
		(\psi, J, K) \arrow[r, squiggly, "g"] & 
		(\gpsi,\gJ,\gK)=(\Ad(g)\circ \psi, (g\ten g)\cdot J\cdot \Delta(g)^{-1}, g\cdot K)\,,
	\end{tikzcd}
\end{equation}
and it extends to the tensor K-matrix as $\gTKM{}=(1 \ot g) \cdot \TKM{}$, see \cite[Sec.~5.10]{AV25b}. \\

In the following, we shall consider gauge transformations of the form $g = D \cdot K^{-1}$ where $D \in A^\times$ is a distinguished group-like element. In this case, the expression of the gauge-transformed cylindrical structure simplifies. 
We have $\gpsi=\Ad(g)\circ\psi$, and
\eq{
	\gJ = \Ad(1\ten D)(R^{\gpsi}), \qq \gK = D, \qq \gTKM{} = (R^{\gpsi})_{21} \cdot (1\ten D)\cdot R.
}

\begin{remarks}\label{rmk:gauge}\hfill
	\begin{enumerate}\itemsep0.25cm
	\item 
	The cylindrical structure $(\gpsi, \gJ, \gK)$ is canonical in the sense that it is a \emph{gauge invariant} of $(\psi,J,K)$: 
	it only depends on the gauge-equivalence class of $(\psi,J,K)$. 
	In other words, for any $g\in A^\times$, one has
	\begin{equation}
		\begin{tikzcd}
			(\psi, J, K) \arrow[r, squiggly, "g"] \arrow[d, squiggly, "D\cdot K^{-1}"']& 
			(\Ad(g)\circ \psi, (g\ten g)\cdot J\cdot \Delta(g)^{-1}, g\cdot K)
			\arrow[d, squiggly, "D\cdot (g \cdot K)^{-1}"]\\
			(\gpsi, \gJ, \gK)\arrow[r, equal]& (\gpsi, \gJ, \gK).
		\end{tikzcd}
	\end{equation}
	\item
	In \cite{AV25b}, we considered the special case $D=1$, which led us to determine a sufficient condition for \eqref{K:support}.
	Namely, let $\xi \coloneqq \Ad(K^{-1}) \circ \psi $ and set
	\eq{
		B_\xi \coloneqq \{ a \in A \, | \, (\xi \ot \id)(\Delta(a)) = \Delta(a) \}. \label{xi:def} 
	}
	Note that $B_\xi$ is the maximal right coideal subalgebra of $A$ contained in the fixed-point subalgebra $A^\xi$. 
	By \eqref{K:intw}, we have $B \subseteq A^\xi$ and hence $B\subseteq B_\xi$. 
	Then, the support condition \eqref{K:support} follows automatically whenever $B$ is maximal, \ie
	\eq{ \label{B:maximal}
		B = B_\xi.
	}
	Indeed, \eqref{K:RE} implies that $(R^\xi)_{21} \cdot R$ lies in $B_{\xi} \ot A$, and therefore so does $(R^\psi)_{21} \cdot K_2 \cdot R$. 
	\item 
	In Section \ref{ss:dualK} we will use this approach to build a gauge-invariant boundary transfer matrix. In particular, we will rely on the coproduct formula of the tensor K-matrix $\gTKM{}$, \ie
	\begin{align}
		\label{gTK:coproduct2} 
		(\id \ot \Delta)(\gTKM{}) &= \gJ^{-1}_{23} \cdot \gTKM{13} \cdot R^{ \gpsi}_{23} \cdot \gTKM{12}
	\end{align}
	as it follows immediately from the definitions.
	\rmkend 
	\end{enumerate}
\end{remarks}

\subsection{Boundary gauge transformation and transfer matrices}
\Omit{
\bartcomment{Here I believe we are setting things up in an awkward way. 
When I was preparing for my talks in Glasgow and Newcastle, it seemed much more natural to me to avoid the boundary gauge transformation and work directly with the cylindrical structure, as follows: \\
1. Lemma that states one has the ring homomorphism provided the twist pair $(\psi,J)$ satisfies $J_{V, W} = (((R_{V^{\psi}, W}^{\t_2})^{-1})^{\t_2})^{-1}$ for all $V,W \in \Mod_{\sf fd}(A)$.\\
2. Lemma that states that in any quasitriangular Hopf algebra the condition on $(\psi,J)$ is algebraic: $J = (1 \ot u) \cdot R^\psi \cdot (1 \ot u^{-1})$.\\
3. Theorem that states that in any balanced Hopf algebra the desired twist pair satisfying the algebraic condition exists in the gauge equivalence class of any given twist pair. \\
4. The proof of the previous theorem immediately permits the $K'$-generalization.\\

On the other hand, the boundary gauge transformation is important as it appears in the formula for the transfer matrix.  
Moreover, the referees seemed OK with our setup. So I will not insist :-). 
}
}

Let $[\Mod_{\sf fd}(A)]$ be the Grothendieck group of the category of finite-dimensional $A$-modules, \ie the abelian group generated by the isomorphism classes $[V]$ with $V \in \Mod_{\sf fd}(A)$ with the relations $[U]+[W]=[V]$ for every short exact sequence $0 \to U \to V \to W \to 0$ in $\Mod_{\sf fd}(A)$. 
Since $A$ is a bialgebra, $[\Mod_{\sf fd}(A)]$ is a ring, with multiplication map given by $[V] \cdot [W] = [V \ot W]$. 
Moreover, it is commutative, since $A$ is quasitriangular and thus $[V\ot W]=[W\ot V]$.

Given a cylindrical structure $(\psi, J, K)$ on $(A,B)$, we aim to construct a ring homomorphism from $[\Mod_{\sf fd}(A)]$ to $B$. Our approach relies on a distinguished family of gauge transformations.

\begin{definition}\label{def:bgt}
An element $g\in A^{\times}$ is called a \emph{boundary gauge transformation} for the twist pair $(\psi,J)$ if, for any $V,W\in\Mod_{\sf fd}(A)$,
\begin{equation}\label{gJR-equation}
	\gJ_{V, W}=(((R_{V^{\gpsi}, W}^{\t_2})^{-1})^{\t_2})^{-1}
\end{equation}
\ie $(g_V \ot g_W)\cdot J_{V,W} \cdot g^{-1}_{V \ot W}  = \Ad(g_V\ten 1)(((R_{V^{\psi}, W}^{\t_2})^{-1})^{\t_2})^{-1}$.
\end{definition}

Definition~\ref{def:bgt} is tailored to the following result.

\begin{theorem} \label{thm:t:Grothendieck}
Let $A$ be a quasitriangular bialgebra, $B \subseteq A$ a right coideal subalgebra, $(\psi,J,K)$ a cylindrical structure on $(A,B)$, and $g\in A^\times$ a boundary gauge transformation for $(\psi,J)$. 
Set $\gTKM{}\coloneqq(1\ten g)\cdot\TKM{}$. Then, there is a ring homomorphism $\TM\colon[\Mod_{\sf fd}(A)] \to B$ given by the assignment
	\eq{
		[V]\mapsto \TM^{(V)} \coloneqq \Tr_2 \big( \gTKM{\bullet,V} \big)
	}
where $V\in\Mod_{\sf fd}(A)$. 
In particular, $[\TM^{(V)}, \TM^{(W)}] = 0$ for all $V,W \in \Mod_{\sf fd}(A)$.
\end{theorem}

\noindent
We refer to $\TM^{(V)}$as the \emph{boundary transfer matrix}.
Note that $\eps(\TM^{(V)}) = \Tr (g_V \cdot K_V)$.

\begin{remark}
The existence of a boundary gauge transformation is a property of the twist pair $(\psi, J)$, since it corresponds to the latter being gauge-equivalent with any twist pair $(\ol{\psi}, \ol{J})$ satisfying \eqref{gJR-equation}. Whenever $(\psi, J)$ is part of a cylindrical structure, the choice of the boundary gauge transformation yields the boundary transfer matrix $\TM$ in terms of the gauge-transformed tensor K-matrix.\rmkend
\end{remark}

\begin{proof}[Proof of Theorem \ref{thm:t:Grothendieck}]
	Let $V,W \in \Mod_{\sf fd}(A)$ with $V\simeq W$. Then, by cyclicity of the trace, we have 
	$\TM^{(V)} = \TM^{(W)}$. 
	Moreover, for any short exact sequence $0 \to U \to V \to W \to 0$ in $\Mod_{\sf fd}(A)$, by additivity of the trace, we have $\TM^{(V)}=\TM^{(U)}+\TM^{(W)}$.
	It remains to prove that $\TM$ is multiplicative, \ie $\TM^{(V \ot W)} = \TM^{(V)} \cdot \TM^{(W)}$. 
	By definition, we have
	\eq{
		\TM^{(V \ot W)} = \Tr_2 \big( \gTKM{\bullet,V \ot W} \big) = \Tr_{2,3} \big( (\id \ot \Del)(\gTKM{})_{\bullet,V,W} \big)
	}
	where $\Tr_{2,3} = \Tr_2 \circ \Tr_3$ with the linear map $\Tr_3: B \ot \End_\bsF(V) \ot \End_\bsF(W) \to B \ot \End_\bsF(V)$ defined analogously.
	By \eqref{gTK:coproduct2}, we obtain
	\begin{align}
		\TM^{(V \ot W)} 
		&= \Tr_{2,3} \Big( (\gJ_{V,W}^{-1})_{23} \cdot (\gTKM{\bullet,W})_{13} \cdot \big( R_{V^{\ol{\psi}},W} \big)_{23} \cdot (\gTKM{\bullet,V})_{12} \Big) \\
		&= \Tr_{2,3} \Big( \big( (\gJ_{V,W}^{-1})_{23} \cdot (\gTKM{\bullet,W})_{13} \big)^{\t_3} \cdot \big( \big( R_{V^{\ol{\psi}},W} \big)_{23} \cdot (\gTKM{\bullet,V})_{12} \big)^{\t_3} \Big) \hspace{-5pt} \\
		&= \Tr_{2,3} \Big( (\gTKM{\bullet,W})_{13}^{\t_3} \cdot (\gJ_{V,W}^{-1})_{23}^{\t_3} \cdot \big( R_{V^{\ol{\psi}},W} \big)_{23}^{\t_3} \cdot 
		(\gTKM{\bullet,V})_{12} \Big) 
	\end{align}
	where we used \eqref{linalg}.
	Since $g$ is boundary, by \eqref{gJR-equation} we have
	\begin{align}
\TM^{(V \ot W)} = \Tr_{2,3} \Big( (\gTKM{\bullet,W})_{13}^{\t_3} \cdot (\gTKM{\bullet,V})_{12} \Big) = \Big( \Tr_2 \gTKM{\bullet,W}^{\t_2} \Big) \cdot \Big( \Tr_2 \gTKM{\bullet,V} \Big) = \TM^{(W)} \cdot \TM^{(V)}, 
	\end{align}
	where we used the first identity of \eqref{linalg} again.
	Since $\TM^{(V \ot W)} = \TM^{(W \ot V)}$, the result follows.
\end{proof}

\section{Boundary transfer matrices for balanced Hopf algebras} \label{sec:transfermatrices:balanced}

In this section, we further assume $A$ to be a balanced Hopf algebra. 
We then provide several examples of Theorem~\ref{thm:t:Grothendieck} by exhibiting explicit boundary gauge transformations.

\subsection{Balanced and sovereign Hopf algebras}\label{ss:balance-sovereign}
Let $A$ be a quasitriangular Hopf algebra with invertible antipode $S\colon A \to A$. 
Then, one has
\eq{ \label{R:antipode}
	(S \ot \id)(R) = R^{-1} \qq\mbox{and}\qq (\id \ot S)(R^{-1}) = R.
}
The invertible element $u\coloneqq m((S^{-1} \ot \id)(R))$, where $m\colon A \ot A \to A$ denotes the multiplication map, is called the \emph{Drinfeld element}.
It satisfies the identities\footnote{The element $u$ is $u_2^{-1}$ in \cite{Dri90}.}
\eq{
	\Ad(u) = S^2 \qq\mbox{and}\qq \Delta(u) = (R_{21} \cdot R)^{-1} \cdot (u \ot u)\,.
}
The Hopf algebra $A$ is called
\begin{itemize}\itemsep1mm
	\item \emph{balanced}  if there exists a central element $b \in A^\times$, called \emph{balance}, such that $\Delta(b) =  (b \ot b) \cdot R_{21} \cdot R$\,;
	\item \emph{sovereign} if there exists a group-like element $D\in A^\times$ such that $\Ad(D) = S^2$.
\end{itemize}
It is clear that $A$ is balanced if and only if it is sovereign, since (up to group-like central elements) the balance and the sovereign element satisfy $D=b\cdot u$.

\begin{remark}\hfill
	\begin{itemize}\itemsep1mm
		\item
		The two definitions are essentially implicit in \cite{Dri90}. 
		A balance, as presented above, is a  weaker version of a ribbon as considered in \cite{RT90}. 
		For the terminology of sovereign Hopf algebras, see \cite{Bic01}. 
		\item 
		A balance is also a K-matrix. Namely, if $b$ is a balance, then $(\psi,J,K) = (\id,R_{21}^{-1},b)$ is a cylindrical structure on $(A,A)$.
By taking $g=b^{-1}$ we see that it is gauge-equivalent to $(\id,R,1)$.
		\item A half-balance is an element $h \in A^\times$ such that $\Delta(h) = (h \ot h) \cdot R$ and $h^2$ is central (and hence a balance) \cite{KT09, ST09}. 
		Then a cylindrical structure on $(A,A)$ is given by $(\psi,J,K) = (\Ad(h),1 \ot 1,h)$. 
		By taking $g=h^{-1}$, we see that it is gauge-equivalent to $(\id,R,1)$.
		\rmkend
	\end{itemize}
\end{remark}

\subsection{R-matrices and dual modules}
Henceforth, we assume that $A$ is balanced, with balance $b$ and sovereign element $D$.
Let $V$ be an $A$-module. 
The algebraic dual $V^\vee$ is naturally equipped with two $A$-module structures. 
Namely, the right dual module $V^*$ and the left dual module ${}^*V$ are defined by the assignments
\eq{ \label{Hopf:dualmodule}
	a_{V^*} = (S(a)_V)^{\t}, \qq a_{{}^*V} = (S^{-1}(a)_V)^{\t},
}
respectively, where $a\in A$ and ${}^{\t}$ denotes the dual of a linear map.
For our purposes, we will consider only left dual modules.
For all $A$-modules $V,W$ one has $R_{V,{}^*W}=(R_{V,W}^{-1})^{\t_2}$
by \eqref{R:antipode}, and, similarly,
\eq{
	R_{V,{}^{**}W}=(((R_{V,W}^{-1})^{\t_2})^{-1})^{\t_2}.
}

When $V$ is finite-dimensional, the action of the sovereign element provides a canonical isomorphism of modules 
$V\to {}^{**}V$. 
Thus, we have, if $\dim(W)<\infty$,
\begin{align}
	\Ad(1 \ot D^{-1})(R)_{V,W} &= (((R_{V,W}^{-1})^{\t_2})^{-1})^{\t_2} 
	\label{R:doubledual2}
\end{align}

\subsection{Boundary transfer matrix for balanced Hopf algebras}\label{ss:bound-bal}
In the case of balanced Hopf algebras, the defining property of a boundary gauge transformation simplifies and becomes algebraic.

\begin{lemma}\label{lemma-gJR}
Let $g\in A^\times$, $(\psi, J)$ a twist pair, and $(\gpsi, \gJ)$ the gauge-transformed twist pair defined in \eqref{cylindrical:gauge}. Then $g$ is a boundary gauge transformation if 
\begin{equation}\label{gJR-equation-balance}
	\gJ=\Ad(1\ten D)(R^{\gpsi})
\end{equation}
\ie $(g \ot g)\cdot J \cdot \Delta(g)^{-1}  = \Ad(g\ten D)(R^{\psi})$.
\end{lemma}

\begin{proof}
We can rewrite \eqref{R:doubledual2} as 
\begin{equation}
\label{R:doubledual2-bis}  (((R^{\t_2}_{V,W})^{-1})^{\t_2})^{-1} = \Ad(1 \ot D)(R)_{V,W}\,.
\end{equation}
This implies that $g$ is a boundary gauge transformation if and only if
\begin{align}\label{gJR-equation-balance-rep}
	\gJ_{V, W}=(((R_{V^{\gpsi}, W}^{\t_2})^{-1})^{\t_2})^{-1}=\Ad(1\ten D)(R)_{V^{\gpsi}, W},
\end{align}	
as required.
\end{proof}

\begin{remark}
More precisely, \eqref{gJR-equation-balance-rep} is equivalent to
\eqref{gJR-equation-balance} whenever the category $\Mod_{\sf fd}(A)$ separates points in $A$, see Remark~\ref{rmk:separation}.\rmkend	
\end{remark}

The most obvious examples of boundary transfer matrix maps arise from the following boundary gauge transformations.

\begin{prop}\label{prop:ubgt}
Let $A$ be a balanced Hopf algebra, $B\subseteq A$ a right coideal subalgebra, and $(\psi,J,K)$ a cylindrical structure on $(A,B)$ with tensor K-matrix $\TKM{}$.
\begin{enumerate}\itemsep1mm
	\item The Drinfeld element $u$ is a boundary gauge transformation for $(\psi,J) = (\id,R_{21}^{-1})$. Then, $\gTKM{}=(1\ten u)\cdot\TKM{}$ yields a boundary transfer matrix.
	\item The inverse of the balance $b^{-1}$ is a boundary gauge transformation for $(\psi,J) = (\Ad(D),R_{21}^{-1})$. Then, $\gTKM{}=(1\ten b^{-1})\cdot\TKM{}$ yields a boundary transfer matrix.
\end{enumerate}
\end{prop}

\begin{proof}	
In both cases we have $\gpsi=\Ad(D)$ and $\gJ=R$. Then, it is enough to observe that 
$R=\Ad(D\ten D)(R)=\Ad(1\ten D)(R^{\gpsi})$.
\end{proof}


\subsection{Dual K-matrix}\label{ss:dualK}

\nc{\tTKM}[1]{\wt{\bbK}_{#1}} 


We now describe the second example of a boudary transfer matrix map, which will be the most relevant for applications in quantum affine algebras (see Section~\ref{sec:affine}).

\begin{prop}\label{prop:wtKbg}
Let $A$ be a balanced Hopf algebra, $B\subseteq A$ a right coideal subalgebra, and  $(\psi,J,K)$ a cylindrical structure on $(A,B)$ with tensor K-matrix $\TKM{}$. Then, $\wt K\coloneqq D\cdot K^{-1}$ is a boundary gauge transformation and $\gTKM{}=(1\ten \wt K)\cdot\TKM{}$ yields a boundary transfer matrix.
\end{prop}

\begin{proof}
First, we observe that
\eq{
	\Delta(\wt K) = \wt K_1 \cdot \Ad(1 \ot D)(R^{\psi})^{-1} \cdot \wt K_2 \cdot J\,. \label{tildeK:coproduct}
}
Set $g=\wt K$. Then, $\gpsi=\Ad(\wt K)\circ \psi$, and 
\eq{	
	\gJ = (\wt K \ot \wt K) \cdot J \cdot \Delta(\wt K)^{-1} =\wt K_1 \cdot \Ad(1\ten D)(R^{\psi}) \cdot \wt K_1^{-1} =\Ad(1\ten D)(R^{\gpsi})\,.
}
Thus, the result follows from Lemma~\ref{lemma-gJR}.	
\end{proof}

\begin{remark}
As observed in Section~\ref{ss:gauge}, the cylindrical structure $(\gpsi, \gJ, \gK)$, obtained by gauging with $\wt K$, is given by $\gpsi=\Ad(\wt K) \circ \psi$ (so $\gpsi|_B = \Ad(D)|_B$), $\gJ=\Ad(1\ten D)(R^{\gpsi})$, and $\gK=D$, while the tensor K-matrix is 
\begin{equation}
	\gTKM{} \coloneqq (1 \ot \wt K) \cdot \TKM{}  = (R^{\ol \psi})_{21} \cdot (1 \ot D) \cdot R \,.
\end{equation}
Note that both $\gTKM{}$ and $\TM^{(V)}=\Tr_2(\gTKM{\bullet, V})$ are gauge invariants of $(\psi,J,K)$ in the sense of Remark~\ref{rmk:gauge} (2). 
\rmkend
\end{remark}

\subsection{Generalized dual K-matrices} \label{ss:gen-dualK}
Proposition \ref{prop:wtKbg} admits an immediate generalization.

\begin{prop}\label{prop:gen-wtKbg}
Let $A$ be a balanced Hopf algebra, $B \subseteq A$ right coideal subalgebras, $(\psi,J,K)$ a cylindrical structure on $(A,B)$ with tensor K-matrix $\TKM{}$, and $K'\in A^\times$ such that
\begin{equation}
	\Delta(K') = J^{-1} \cdot K'_2 \cdot R^\psi \cdot K'_1\,.
\end{equation}
Then, $\wt K\coloneqq D\cdot (K')^{-1}$ is a boundary gauge transformation for $(\psi,J)$
and yields a new boundary transfer matrix
$\TM^{(V)}=\Tr_2((1\ten \wt K)\cdot\TKM{})$.
\end{prop}


\begin{proof}
The proof is identical to that of Proposition~\ref{prop:wtKbg}.
\end{proof}
\begin{remarks} \label{rmk:transfermatrices} \hfill
	
	\begin{enumerate}\itemsep1mm
		\item Both examples described in Proposition~\ref{prop:ubgt} are special cases of the construction above. 
	Recall that $(\id, R_{21}^{-1}, b)$ is a cylindrical structure on $(A,A)$ and hence so is $(\Ad(D), R_{21}^{-1}, Db)$. 
	Therefore, if $(\id, R_{21}^{-1}, K)$ is a cylindrical structure on $B$, we obtain $\wt K=Db^{-1}=u$ or $\wt K=Db^{-1}D^{-1}=b^{-1}$ as a boundary gauge transformation by choosing $K'=b$ or $K'=Db$, respectively.
	\item 
	Whenever $K'$ is part of a cylindrical structure of the form $(\psi, J, K')$ for some coideal subalgebra $B'\subseteq A$, the construction above is reminiscent of Sklyanin's \emph{two-boundary} transfer matrix \cite{Skl88}, see Section \ref{sec:quantumintegrability}.
	As before, the resulting boundary transfer matrix map $\TM^{(V)}=\Tr_2(\gTKM{\bullet, V})$ is invariant under gauge transformations acting simultaneously on both $(\psi,J,K)$ and $(\psi, J, K')$. 
	
	Moreover, the image of the corresponding boundary transfer matrix map $\TM \colon [\Mod_{\sf fd}(A)]\to B$ is contained in $B\cap S(B')$. 
	The proof relies essentially on the same arguments used in \cite[Prop.~3]{Skl88}. 
	When $B'=B$, as in Section~\ref{ss:dualK}, this condition may be very restrictive. 
	For instance, in the case of quantum groups of finite type, it implies that $\TM$ is rather trivial, see Section~\ref{sec:boundarytransfermatrices:qgpfinite}. 
	It is therefore necessary to consider boundary transfer matrix maps with $B'\neq B$ in order to get nontrivial elements. For instance, the boundary transfer matrix map associated to the Drinfeld element recovers the morphism constructed by Kolb in \cite{Kol20},
	which surjects over the center of a quantum symmetric pair subalgebra, see Remark~ \ref{rmk:kolb-gendual} (\ref{rmk:kolb}). In contrast, this trivialization does not occur in the case of transfer matrices depending on a spectral parameter for quantum affine algebras, as we observe in Section \ref{sec:quantumintegrability}.
	\item 
Although we had to carefully select the cylindrical structure, Proposition \ref{prop:wtKbg} can be seen as the boundary counterpart to the following statement about universal transfer matrices defined in terms of a quasitriangular bialgebra $A$ universal R-matrix $R$. 
Namely, a ring homomorphism $[\Mod_{\sf fd}(A)] \to A$ is defined by the assignment
\[
[V] \mapsto \TM^{(V)}_{\sf per} = \Tr_V R_{\bullet,V} \in A
\]
with $R_{\bullet,V}$ also known as \emph{(universal) L-operator}.
This construction finds its origin in the transfer matrices for vertex models and quantum spin chains with periodic boundary conditions, see, \eg \cite{Ba82}, and plays a key role in the definition of q-characters of quantum affine algebras, see \cite{FR99,FM01}.
It permits the following generalization:
\[
[V] \mapsto \TM^{(V)}_{{\sf per},a} = \Tr_V ((1 \otimes a_V) \cdot R_{\bullet,V}) \in A
\]
where $a \in A^\times$ is any group-like element.
Proposition \ref{prop:gen-wtKbg} can be viewed as the boundary counterpart of this generalization.
\rmkend
	\end{enumerate}
\end{remarks}


\subsection{Completions}\label{ss:completions}
The purely algebraic setting described above is too restrictive, and does not include R-matrices and K-matrices arising from quantum groups. Indeed, these operators are not algebraic, but rather {\em topological}, \ie they belong to a suitable completion of the algebra, which is representation-theoretic in nature.

More precisely, for any 
algebra $A$ and any a full subcategory $\cC \subseteq \Mod(A)$, we consider the 
algebra $\End(\FF{\cC}{})$ of endomorphisms of the forgetful functor $	\FF{\cC}{}: \cC \to \vect$. An element $f\in\End(\FF{\cC}{})$ is a $\cC$-tuple of linear maps $f=(f_V: V \to V)_{V \in \cC}$ such that $\iota \circ f_V = f_W \circ \iota$ for all intertwiners $\iota: V \to W$ in $\cC$.
Linear combinations and compositions of such tuples are defined entrywise.

\begin{remark}\label{rmk:separation}
$\End(\FF{\cC}{})$ is usually referred to as the \emph{completion of $A$ with respect to $\cC$}, since there is a natural algebra homomorphism $\iota_{\cC}\colon A \to \End(\FF{\cC}{})$ given by the assignment $a \mapsto (a_V)_{V \in \cC}$. Whenever $\iota_{\cC}$ is injective, $A$ identifies with a subalgebra of $\End(\FF{\cC}{})$, and $\cC$ is said to \emph{separate points in $A$}. 
\hfill\rmkend
\end{remark}

If $A$ is a bialgebra and $\cC \subset \Mod(A)$ is a tensor subcategory then for $n \in \Z_{\ge 0}$ we consider the algebra $\End(\FF{\cC}{n})$ of endomorphisms of the $n$-fold forgetful functor $\FF{\cC}{n}: \cC^n \to \vect$ given by 
\eq{
	\FF{\cC}{n}(V_1,\ldots,V_n) = \FF{\cC}{}(V_1 \ot \cdots \ot V_n).
}
$\End(\FF{\cC}{})$ is not a bialgebra in general, but it is equipped instead with a cosimplicial structure, see, \eg \cite{Dav07,ATL19}. In particular, the coproduct map $\Delta: A \to A \ot A$ lifts to a morphism 
$\Del : \End(\FF{\cC}{}) \to \End(\FF{\cC}{2})$ given by $\Del(f)_{V,W} = f_{V \ot W}$ for all $V,W\in\cC$. 
More generally, let $B$ be a right coideal subalgebra of $A$, $\cD \subseteq \Mod(B)$ a module category over $\cC$, and $\FF{\cD \boxtimes \cC}{} : \cD \boxtimes \cC \to \vect$ the forgetful functor given by $(M,V) \mapsto \FF{\cD}{}(M \ot V)$. 
Then, the coproduct $\Delta: B \to B \ot A$ lifts to a morphism $\Del: \End(\FF{\cD}{}) \to \End(\FF{\cD \boxtimes \cC}{})$.

\begin{remark}
In this setting, the gauge transformation induced by a K-matrix yields a map $\gpsi=\Ad(K)\circ\psi$, which is an automorphism of $\End(\FF{\cC}{})$, but it does not generally preserve $A\subseteq\End(\FF{\cC}{})$.\rmkend
\end{remark}


\Omit{
	
\subsection{Miscellanea}\bartcomment{Fancy ideas and loose ends, to be discussed. I am keen to include the left-right symmetry of $\TM^{(V)}$, especially if we can give some more meaning to $\wt K$.}

Here $\wt R=\Ad(1\ten D)(R^{-1})=(1\ten S)(R)$.

\subsubsection{Action of the antipode on $K$}

To obtain a formula for $S(K)$, it suffices to do this in the case where $\psi$ is a coalgebra antiautomorphism (at least for DJ quantum groups where we have the semistandard gauge). 
So $J= 1 \ot 1$, $\psi \circ S = S^{-1} \circ \psi$. 
Applying $m \circ (\id \ot S)$ to $\Del(K) = K_2 \cdot R^\psi \cdot K_1$ yields
\eq{
S(K) = m\big( (\Ad(K^{-1}) \circ \psi \ot S)(R) \big)^{-1} \cdot K^{-1} = m\big( (R^{\psi'})^{-1} \big)^{-1} \cdot K^{-1} 
}
where
\eq{
\psi' = \Ad(K^{-1}) \circ \psi \circ \Ad(D^{-1}).
}

Applying $S$ to the linear axiom \eqref{K:intw} and rearranging we obtain 
\eq{
S(K) \cdot D \cdot \psi(b) = b \cdot S(K) \cdot D, \qq \text{for all } b \in S(B).
}

\subsubsection{Conceptual formalism for $\wt K$}

We should attempt to conceptually describe the full axiomatic scheme around $\wt K$. 
In particular, what is the natural \emph{left} coideal subalgebra $\wt B \subset A$ that is associated to $\wt{K}$?
Note that $S(B)$ is a left coideal subalgebra.
The following axiomatics follow immediately from the fact that $(\psi,J,K)$ is a cylindrical structure on $(A,B)$:
\begin{align}
\Ad(\wt K) \circ \psi|_B &= S^2|_B, \\
\wt R \cdot \wt K_2 \cdot (R^\psi)_{21}^{-1} &\in B \ot A, \label{tildeK:support} \\
\Delta(\wt K) &= \wt K_1 \cdot \wt R^\psi \cdot \wt K_2 \cdot J,
\end{align}
as well as the generalized reflection equation
\eq{ \label{tildeK:RE}
\wt K_1 \cdot \wt R^\psi \cdot \wt K_2 \cdot (R^{\psi\psi}_{21})^{-1} = R^{-1} \cdot \wt K_2 \cdot (\wt R^\psi)_{21} \cdot \wt K_1 \qq \in A \ot A.
}

This could go as follows: set $\wt B = S(B)$.
From the above relation $\Ad(\wt K^{-1} \cdot D)|_{B} = \psi|_{B}$ one gets, by applying $S^{-1}$, using $\Ad(D) = S^2$,
\eq{
\Ad(S^{-1}(\wt K))|_{\wt B} = S^{-1} \circ \psi \circ S^{-1}|_{\wt B}.
}
Note that, if $\psi$ is a coalgebra antiautomorphism then $\psi \circ S^{-1} = S \circ \psi$. 
If $A$ is a quantum group then $\psi$ is a coalgebra antiautomorphism, up to a factor $\Ad(T_\theta)$.
From
\eq{
\Delta(T_\theta) = (T_\theta \ot T_\theta) \cdot R_\theta
}
we readily obtain $\eps(T_\theta) =1$ and hence $S(T_\theta) = u_{2,\theta} \cdot T_\theta^{-1}$ where $u_{2,\theta}$ is the $\theta$-analogue (\ie corresponding to the parabolic Hopf subalgebra defined by the subdiagram $X$) of the element $u_2 = u_4$ from \cite[Remarks after Cor. 4.2.4]{CP95}.
If we can somehow absorb this factor $u_{2,\theta}$ then $\psi \circ S^{-1} = S \circ \psi$ in general and we get
\eq{
\Ad(S^{-1}(\wt K^{-1}))|_{\wt B} = \psi|_{\wt B}.
}
Is it nicer? Note that $S^{-1}(\wt K^{-1}) = S^{-1}(K^{-1}) \cdot D^{-1} = D^{-1} \cdot S(K^{-1})$.\\

Another possible idea is to take the ``square root'' of $S^2 = \Ad(D)$ and impose instead $S = \Ad(\wt D) \circ \sigma $ where $\sigma$ is a suitable algebra anti-automorphism and $\wt D$ is a suitable invertible element; then we can apply $S$ to the axioms for $K$ and obtain axioms for $\wt K$...?
For $U_q(\fkg)$ we can take $\sigma$ to be the map multiplying $E_i$ and $F_i$ by $-1$ and inverting $\Kg{h}$ and we can take $\wt D$ to be an element in the weight completion of $U_q(\fkg)$ of the form given in \cite[Sec. 4.9]{AV22}.

\subsubsection{Left-right symmetry of the transfer matrix and the subtlety with $B \cap S(B)$}

Independently we can motivate $S(B)$ as follows, which also makes manifest an expected left-right symmetry of the transfer matrix.
Consider the following tensor analogue of $\tilde K$ (see \eqref{tildeK:support}): 
\eq{
\tTKM{} = \wt R \cdot \wt K_2 \cdot (R^\psi)_{21}^{-1} \in B \ot A.
}
One can show that the element $\Del(\TM^{(V)})$, which a priori lies in $B \ot A$, actually satisfies
\eq{ \label{transfermatrix:coproduct}
\Del(\TM^{(V)}) = \Tr_3 \big( (S \ot \id)(\tTKM{})_{\bullet V} \big)_{23} \cdot (\TKM{\bullet,V})_{13} 
}
and hence lies in $B \ot S(B)$. 
Applying the counit to the first leg of \eqref{transfermatrix:coproduct}, we get $\TM^{(V)} \in S(B)$. Hence we would obtain the following:
\eq{
\TM^{(V)} \in B \cap S(B).
}
If $(A,B)$ is a quantum symmetric pair, this is a rather small subalgebra, containing $U_q(\g_X) U_q(\fkh^\phi)$ and it is unclear if it is bigger than that. 
Indeed, if $A = U_q(\fksl_2)$ we can choose $B = \langle F - q^{-1} \gamma E k^{-1} + \sigma k^{-1} \rangle$ and then one obtains $B \cap S(B) = \bsF$. 
However, $\TM^{(V)}$ lies in a completion and perhaps there is no issue at all. \\

The proof of \eqref{transfermatrix:coproduct} is inspired by the proof of \cite[Prop. 3]{Skl88} (refactorization property of two-boundary transfer matrices).
Noting that 
\[
(S \ot S^{-1})(\wt R) = (\id \ot S^{-1})(R), \qq (S \ot S^{-1})((R^\psi)^{-1}_{21}) = (\id \ot S^{-1})((R^\psi)_{21}),
\]
we deduce
\begin{align}
(S \ot \id)(\tTKM{})_{\bullet, V} &= (S \ot S^{-1})(\tTKM{})_{\bullet, V^*}^{\t_2} \\
&= \big( (S \ot S^{-1})((R^\psi)_{21}^{-1}) \cdot S^{-1}(\wt K)_2 \cdot (S \ot S^{-1})(\wt R) \big)_{\bullet, V^*}^{\t_2} \\
&= \big( (\id \ot S^{-1})((R^\psi)_{21}) \cdot S^{-1}(\wt K)_2 \cdot (\id \ot S^{-1})(R) \big)_{\bullet, V^*}^{\t_2} \\
&= \Big( ((R_{V^\psi,\bullet})_{21})^{\t_2} \cdot (\wt K_V)_2^{\t_2} \cdot (R_{\bullet,V})^{\t_2} \Big)^{\t_2}, \\
&= \Tr_3 \big( (\wt K_V)_3 \cdot (R_{V^\psi,\bullet})_{31} \cdot ({\sf flip}_{V,V})_{23} \cdot (R_{\bullet,V})_{13} \big).
\end{align}
Hence,
\[
((S \ot \id)(\tTKM{})_{\bullet, V})_{23} = \Tr_4 \big( (\wt K_V)_4 \cdot (R_{V^\psi,\bullet})_{42} \cdot ({\sf flip}_{V,V})_{34} \cdot (R_{\bullet,V})_{24} \big).
\]
and
\[
({\sf flip}_{V,V})_{34} \cdot (R_{\bullet,V})_{24} \cdot   ( \TKM{\bullet,V} )_{13}  = ( \TKM{\bullet,V} )_{14} \cdot ({\sf flip}_{V,V})_{34} \cdot (R_{\bullet,V})_{24}.
\]
Therefore,
\begin{align}
& \Tr_3 \Big( \big( (S \ot \id)(\tTKM{})_{\bullet, V} \big)_{23} \cdot ( \TKM{\bullet,V} )_{13} \Big) = \\
&= \Tr_{3,4} \Big( (\wt K_V)_4 \cdot (R_{V^\psi,\bullet})_{42} \cdot ( \TKM{\bullet,V} )_{14} \cdot ({\sf flip}_{V,V})_{34} \cdot (R_{\bullet,V})_{24}   \Big) \\
&= \Tr_3 \Big( (\wt K_V)_3 \cdot (R_{V^\psi,\bullet})_{32} \cdot ( \TKM{\bullet,V} )_{13} \cdot (R_{\bullet,V})_{23}   \Big) \\
&= \Tr_3 (\Delta \ot \id) \big( (1 \ot \wt K) \cdot \TKM{\bullet,V} \big)_{\bullet,\bullet,V} \\
&= \Delta \big( \Tr_2 \gTKM{\bullet,V} \big) ,
\end{align}
as required.
}


\section{Quantum groups and R-matrices} \label{sec:quantumgroups}

\subsection{Kac-Moody algebras  \cite{Kac90}}\label{ss:KM}

Let $\g$ be the Kac-Moody algebra defined over $\bbC$ in terms of an indecomposable symmetrizable generalized Cartan matrix $A=(a_{ij})_{i,j\in\IS}$ with $\IS$ a finite set. 
Fix non-negative integers $\{\de{i}\;|\; i\in\IS\}$ such that the matrix $(\de{i}a_{ij})_{i,j\in\IS}$ is symmetric.
Let $\h\subset\g$ be the standard Cartan subalgebra.
Let $\{ e_i,f_i\}_{i \in \IS}$ be the Chevalley generators with corresponding simple roots $\al_i \in \fkh^*$ and simple coroots $h_i = [e_i,f_i] \in \fkh$.
We consider the subalgebras
\[
\g' = [\g,\g] = \C \langle \{ e_i,f_i \}_{i \in \IS} \rangle, \qq \qq \h' = \h \cap \g' = \bigoplus_{i \in \IS} \C h_i.
\] 
Let $\iip{\cdot}{\cdot}$ be the nondegenerate invariant bilinear form on $\g$ depending on a chosen linear complement $\h''$ of $\h'$ in $\h$.
The form $\iip{\cdot}{\cdot}$ induces a nondegenerate pairing between the opposite subalgebras in the triangular decomposition $\g \cong \fkn^- \oplus \fkh \oplus \fkn^+$, where $\fkn^+ \coloneqq  \bbC\langle \{ e_i \}_{i \in \IS} \rangle$, $\fkn^- \coloneqq  \bbC\langle \{ f_i \}_{i \in \IS} \rangle$, and we obtain a nondegenerate pairing between the subalgebras $\fkb^\pm \coloneqq  \langle \fkn^\pm, \fkh \rangle$.

We consider the root and coroot lattices:
\eq{
\Qlat = \Sp_\Z \{ \al_i \}_{i \in \IS} \subset \h^*, \qq \Qlat^{\vee} = \Sp_\Z \{ h_i \}_{i \in \IS}  \subset \h.
}
Having chosen a basis $\{d_{1},\dots, d_{\corank(A)}\}$ of $\h''$ such that $\al_i(d_r) \in \Z$ for all $i \in \IS$, $1 \le r \le \corank(A)$, we define the \emph{extended} coroot lattice and hence the weight lattice:
\[
\Qvext=\Qlatv \oplus \Sp_\Z \{ d_{1},\dots, d_{\corank(A)} \} \subset{\h}
\aand
\Pext=\{\lambda\in{\h}^*\;|\;\lambda(\Qvext)\subset\bbZ\}.
\]
The bilinear form on $\fkh^*$ dual to $\iip{\cdot}{\cdot}$ will be denoted by the same symbol.
Note that its restriction to $\Pext \times \Pext$ takes values in $\frac{1}{m}\Z$ for some $m \in \Z_{>0}$.

\subsection{Drinfeld-Jimbo quantum groups \cite{Jim85,Dri86,Lus94}}\label{ss:quantum-group}

Denote by $\bsF$ the algebraic closure of $\bbC(q)$ where $q$ is an indeterminate.
The quantum Kac-Moody algebra associated to ${\g}$ is the unital associative $\bsF$-algebra $\Uqg$ with generators $\Eg{i}$ and $\Fg{i}$ ($i\in\IS$), and $\Kg{h}$ ($h\in\Qvext$) subject to:
\eq{
\qq \qq
\left.
\begin{gathered}
	\Kg{h}\Kg{h'}=\Kg{h+h'}, \qq \Kg{0}=1, \\	
	\Kg{h}\Eg{i}=q^{\cp{\rt{i}}{h}}\Eg{i}\Kg{h}, \qq \Kg{h}\Fg{i}=q^{-\cp{\rt{i}}{h}}\Fg{i}\Kg{h},\\
 [\Eg{i},\Fg{j}]=\drc{ij}\frac{\Kg{i}-\Kg{i}^{-1}}{q_i-q_i^{-1}}, \label{Uqag:relns3} 
\end{gathered}
\right\}
{ \text{ for all } i,j\in\IS, \atop h,h'\in\Qvext,}
}
where $q_i \coloneqq  q^{\de{i}}$, $\Kg{i}^{\pm1} \coloneqq  \Kg{\pm \de{i}\cort{i}}$,
together with the quantum Serre relations (see, \eg \cite[3.1.1 (e), Cor.~33.1.5]{Lus94}).
We consider the Hopf algebra structure on $\Uqg$ uniquely determined by the following coproduct formulae
\eq{
\Delta(\Eg{i}) = \Eg{i}\ten1+\Kg{i}\ten\Eg{i}, \qu \Delta(\Fg{i}) = \Fg{i}\ten\Kg{i}^{-1}+1\ten\Fg{i}, \qu \Delta(\Kg{h}) = \Kg{h}\ten\Kg{h}, 
}
for any $i\in\IS$ and $h\in\Qvext$.
In particular, the antipode $S$ is defined by
\eq{
\label{Uqg:S}	S(\Eg{i})	= -\Kg{i}^{-1}\Eg{i},	\qq  S(\Fg{i}) = -\Fg{i}\Kg{i}, \qq S(\Kg{h}) = \Kg{h}^{-1}.
}
We consider the usual triangular decomposition $\Uqg = \Uqnm \cdot \Uqh \cdot \Uqnp$ in terms of the subalgebras $\Uqnp$, $\Uqh$, $\Uqnm$ generated by $\{ E_i \}_{i \in \IS}$, $\{ \Kg{h} \}_{h\in\Qvext}$ and $\{ F_i  \}_{i \in \IS}$, respectively.

\subsection{Modules over $\Uqg$} \label{sec:Uqgcategories}
Let $V$ be an $\Uqh$-module. For any $\mu\in\Pext$, the corresponding weight space in $V$ is 
\[V_\mu\coloneqq\{v\in V\,\vert\,\forall h\in\Qvext,\,\,\Kg{h}\,v=q^{\mu(h)}v\}.\]
A $\Uqg$--module $V$ is a (type $\bf 1$) {\it weight module} if $V=\bigoplus_{\mu\in\Pext} V_\mu$.
Such a weight module $V$ is said to be
\begin{itemize}\itemsep1mm
	\item\label{cond:int} {\it integrable} if all $E_i,F_i$ act locally nilpotently;\footnote{This implies that $V$ is completely reducible as a (possibly infinite) direct sum of simple finite--dimensional modules over the $U_q(\fksl_2)$-subalgebra generated by $E_i, \Kg{i}^{\pm1}, F_i$, for each $i \in \IS$.}
	\item\label{cond:O} in {\it category $\cO$} if the action of $\Uqnp{}$ is locally finite.\footnote{
		This implies in particular that, for any $v\in V$, $\Uqnp_{\beta} v=0$ for all but finitely many $\beta\in\sfQ^+$. Therefore, $\cO$ coincides with the category $\cC^{\operatorname{hi}}$ used in \cite[Sec.~3.4.7]{Lus94}.}
\end{itemize}
We denote by $\WUqg$, $\OUqg$, $\OintUqg$ the full subcategories of weight, category $\cO$, integrable category $\cO$ modules in $\Mod(\Uqg)$, respectively. 
The category $\WUqg$ is only monoidal, while the category $\OUqg$
is braided, see Section~\ref{sec:uniR}. The category $\OintUqg$ is a semisimple braided subcategory of $\OUqg$, whose simple modules are classified 
by dominant integral weights \cite[Thm~6.2.2, Cor.~6.2.3]{Lus94}. 
Note that $\Oint$ contains the category $\Mod_{\sf fd}(\Uqg)$ of finite-dimensional (type-$\bf 1$) $\Uqg$-modules, which admits nontrivial objects only if $\dim(\fkg)<\infty$.

By \cite[Thm.~3.1]{ATL24}, $\Uqg^{\ot n}$ embeds into both $\End(\FF{\cO}{n})$ and $\End(\FF{\cO_{\sint}}{n})$.

\subsection{The universal R-matrix \cite{Dri86, Lus94}} \label{sec:uniR}
The universal R-matrix of $U_q(\fkg)$ is the canonical element of the pairing between the two Hopf subalgebras $U_q(\fkb^\pm) \coloneqq  U_q(\fkh) \cdot U_q(\fkn^\pm)$.
It corresponds to an element  $R \in \End(\FF{\cO}{2})$, which factorizes as $R = \kappa \cdot \Xi$ where
\begin{itemize}\itemsep0.25cm
	\item 
	$\kappa = (\kappa_{V,W})_{V,W \in \cW} \in \End(\FF{\cW}{2})$ is a diagonal operator, which 
	acts on $V_\la \ot W_\mu$ as multiplication by $q^{(\la,\mu)}$\,;
	\item 
	$\Xi \in \End(\FF{\cO}{2})$ is the quasi R-matrix, \ie it is the weight-zero operator of the form
	$\Xi = \sum_{\la \in \Qlat^+} \Xi_\la$ where  $\Xi_\la \in U_q(\fkn^-)_{-\la} \ot U_q(\fkn^+)_\la$ is the canonical element of the pairing between $U_q(\fkn^-)_{-\la}$ and $U_q(\fkn^+)_\la$. 
\end{itemize}
The identity \eqref{R:intw} is satisfied in $\End(\FF{\cO}{2})$ and the identities \eqref{R:coproduct} in $\End(\FF{\cO}{3})$.
The assignment $(V,W) \mapsto \text{flip}_{V,W} \cdot R_{V,W}$ defines a braiding on $\cO$ and $\Oint$.

\subsection{Sovereign element}
Choose $\rho \in \Plat$ so that $\rho(h_i)=1$ for all $i \in \IS$.
Let $D = (D_V)_{V \in \cO} \in \End(\FF{\cO}{})$ be the operator defined by 
\eq{ \label{D:def}
D_V \cdot v = q^{-2(\rho,\mu)} v\,
}
for any $V\in\cO$, $\mu\in\Plat$, and $v\in V_\mu$.
Note that $\Delta(D) = D \ot D$, $\Ad(D)$ preserves $\Uqg \subset \End(\FF{\cO}{})$ and $\Ad(D) = S^2$. 
Thus, $D$ can be viewed as a sovereign element for $\Uqg$. 
The Drinfeld element $u$ also belongs to $\End(\FF{\cO}{})$. 
Hence, by Section~\ref{ss:balance-sovereign}, the element $b\coloneqq Du^{-1}\in\End(\FF{\cO}{})$ is central and it can be viewed as a balance for $\Uqg$. 



\section{Quantum symmetric pairs and K-matrices} \label{sec:QSPK}
We provide a brief summary of the theory of quantum symmetric pairs and their K-matrices, and survey some recent works on related topics.
The main references for proofs and further details in this section are \cite{Kol14, RV22, AV22, AV25b}.

In 	\cite{AV22, AV25b}, we constructed a family of (basic and tensor) K-matrices arising from quantum symmetric pairs, whose action is defined on $\Oint$ and, in fact, on modules in $\cO$ satisfying a weaker integrability condition. 
Henceforth, for simplicity of the exposition, we consider only the category $\Oint$, and we set
\begin{equation}
	\FF{}{} \coloneqq \FF{\Oint}{}, \qq \FF{}{n} \coloneqq \FF{\Oint}{n}. 
\end{equation}

\subsection{Pseudo-involutions and their q-deformations} \label{sec:pseudoinv} 

A Lie algebra automorphism $\theta: \fkg \to \fkg$ is said to be \emph{of the second kind} if $\theta(\fkb^+)$ is obtained by the adjoint action of the Kac-Moody group on the \emph{opposite} Borel subalgebra $\fkb^-$; this is equivalent to $\dim(\fkn^+ \cap \theta(\fkn^+)) < \infty$, see, \eg \cite[4.6]{KW92}.
If $\theta$ is an \emph{involution} of the second kind then, up to conjugation by an inner automorphism, we may assume $\theta(\fkh)=\fkh$ and $\theta$ fixes pointwise any $\theta$-stable root space, see \cite[Ch.~5]{KW92}. 
The main example is the Chevalley involution $\om: \fkg \to \fkg$, defined by
\eq{
\om(e_i) = -f_i, \qq \om(f_i) = -e_i, \qq \om|_\fkh = -\id_\fkh.
}

A \emph{pseudo-involution} naturally generalizes such an involution in the sense that it is only assumed to be an involution on a stable Cartan subalgebra.
Relying on inner conjugacy again, we may define a pseudo-involution as an automorphism $\phi\colon \g\to\g$ of the second kind such that $\phi(\fkh)=\fkh$, $\phi|_{\fkh}$ is involutive and $\phi$ fixes pointwise any $\phi$-stable root space, cf.~\cite[Sec.~2]{RV22}.
In this case we have $\phi(\fkg_\al) = \fkg_{\phi^*(\al)}$ for all roots $\al$ and from now on we denote the dual map $\phi^*: \fkh^* \to \fkh^*$ also by $\phi$.

Given a pseudo-involution $\phi$ and a multiplicative character $\chi: \Qlat \to \C^\times$ such that $\chi(\al)=1$ whenever $\phi(\al)=\al$, we obtain another pseudo-involution 
$\Ad(\chi) \circ \phi$, 
where $\Ad(\chi)|_{\fkg_\al}$ is the multiplication by $\chi(\al)$.\\

Because of the semidirect product factorization of $\Aut_{\sf Lie}(\fkg)$ given in \cite[4.23]{KW92}, each pseudo-involution $\phi: \fkg \to \fkg$ has a canonical expression 
\eq{ \label{pseudoinv:factorization}
\phi = \Ad(\chi) \circ \Ad(\wt w_X)  \circ \tau \circ \om
}
where $\tau$ is a diagram automorphism\footnote{Recall that a \emph{diagram automorphism} (a bijection $\tau: \IS \to \IS$ such that $a_{\tau(i)\tau(j)} = a_{ij}$ for all $i,j \in \IS$) acts as a Lie algebra automorphism of $\fkg'$ by relabelling: $\tau(e_i) = e_{\tau(i)}$, $\tau(f_i) = f_{\tau(i)}$. 
By \cite[4.19]{KW92}, there exists a complementary subspace $\fkh''$ of $\fkh' = \fkh \cap \fkg'$ such that the action of $\tau$ can be extended to $\fkg = \fkg' \oplus \fkh''$.}, 
$\Ad(\wt w_X)$ is the triple exponential action induced by the longest element $w_X$ of a unique subdiagram $X \subseteq \IS$ of finite type and a character $\chi: \Qlat \to \C^\times$ such that $\chi(\al_i)=1$ if $i \in X$.
Note that $(X,\tau,\chi)$ are uniquely determined by $\phi$ from the expression \eqref{pseudoinv:factorization}. In particular, $X = \{ i \in \IS \, | \, \phi(\al_i) = \al_i \}$ (see \cite[Sec.~2]{RV22} and \cite[Sec.~6.1 \& 6.2]{AV22} for details). The case $X = \emptyset$ is referred to as \emph{quasi-split}. \\

In order for $\phi$ to lift to an algebra automorphism $\phi_q: \Uqg \to \Uqg$, it suffices that the extension of $\tau$ to $\fkh''$ described by \cite[4.19]{KW92} permutes the $\Z$-basis of $\Qvext$, constraining the choice of $d_s$ ($1 \le s \le \cork(A)$).
In the affine case, a suitable choice of $\Z$-basis can always be made, see \cite[Prop.~2.12]{Kol14}, although not in general if $\cork(A)=1$, see \cite[Sec.~3.15]{AV22}.

In \cite[Sec.~6.7]{AV22} a definition for the lift $\phi_q: \Uqg \to \Uqg$ is given, which can be related to the original constructions in \cite{Let99,Kol14}. 
Note that here we include into $\phi_q$ the adjoint action of a lift of $\chi$ to a multiplicative character $\chi_q: \Qlat \to \bsF^\times$ such that $\chi_q(\al_i)=1$ if $i \in X$.
In terms of the formalism of unrestricted specialization, see, \eg \cite{DCK90}, the automorphism $\phi_q$ specializes to the automorphism $\phi$.

\subsection{Pseudo-fixed-point subalgebras and their q-deformations} \label{sec:QSP}

It is useful to constrain $\phi$ by imposing additional conditions on $X(\phi)$, $\tau(\phi)$, $\chi(\phi)$, leading to the notion of an \emph{(enriched) generalized Satake diagram}, see \cite[Def.~3.4]{RV22} and \cite[Sec.~6.4]{AV22}. 
The resulting conditions on $\phi$ are equivalent to requiring that the associated \emph{pseudo-fixed-point subalgebra}
\eq{
\fkk \coloneqq  \langle \{ b_i \coloneqq  f_i + \phi(f_i)\}_{i \in \IS} , \{ e_i \}_{i \in X}, \fkh^\phi \rangle 
}
mimics the fixed-point subalgebra of an involution in the sense that $\fkk \cap \fkh = \fkh^\phi$.
The pair $(\fkg,\fkk)$ is called a \emph{pseudo-symmetric pair}.
Writing $\fkn^+_\phi \coloneqq  \fkn^+ \cap \phi(\fkn^-)$, we note the following key property, called \emph{Iwasawa decomposition}, see \cite[Sec.~3.4]{RV22}, 
\eq{ \label{k:Iwasawa}
\fkg = \fkk \oplus \fkh^{-\phi} \oplus \fkn^+_\phi.
} 

It is convenient to allow some extra degrees of freedom and define a modification $U(\fkk)_s$ of $U(\fkk)$ by adding a scalar $s_i \in \C$ to the generator $b_i \in U(\fkk)$ for $i \in \IS \backslash X$. 
In order to define the Letzter-Kolb deformation of $U(\fkk)_s$, we choose a pseudo-involution $\phi$ with a canonical expression \eqref{pseudoinv:factorization} in terms of an associated enriched generalized Satake diagram $(X,\tau,\chi)$.
The induced constraint on $\chi_q$ is conveniently rewritten in terms the tuple $\gamma=(\chi_q(\al_{\tau(i)}))_{i \in \IS}$.
Also, to deduce basic properties of the deformed algebra, we will need a constrained tuple $\sigma \in \bsF^{\IS}$ which specializes to the tuple $s$.
In order to do this, recall the subsets $\Gamma_q, \Sigma_q \subset \bsF^\IS$ defined in \cite[Sec. 6.8]{AV22} and select $\gamma \in \Gamma_q$ and $\sigma \in \Sigma_q$.\footnote{See also Section~\ref{prob:generalizedconstraints}.}

\begin{definition}[{\cite{Let99,Kol14,AV22}}] \label{def:QSP}
With $X$, $\tau$, $\gamma$, $\sigma$ as above, the \emph{QSP subalgebra} is the subalgebra $\QSP \subseteq \Uqg$ generated by the elements
\eq{
B_i \coloneqq  F_i + \phi_q(F_i) + \si_i \Kg{i}^{-1} \qq i \in \IS \backslash X,
}
and the following $\phi_q$-fixed Hopf subalgebras of $\Uqg$:
\eq{
U_q(\g_X) \coloneqq  \bsF\langle E_i,F_i, \Kg{i}^{\pm 1} \,\vert\, i \in X \rangle, \qq 
U_q(\fkh^\phi) \coloneqq  \bsF\langle \Kg{h} \,\vert\, h \in (\Qvext)^\phi \rangle.
}
The pair $(\Uqg,\QSP)$ is called a \emph{quantum (pseudo-)symmetric pair}, in short \emph{QSP}, and the entries of $\gamma$ and $\si$ are called \emph{QSP parameters}.
\end{definition}

\begin{remark}
As a standalone algebra, $\QSP$ is also called an \emph{$\imath$-quantum group}, see, \eg \cite{BW18b,BW21,LW21,WZh23,Wa24}.\rmkend
\end{remark}

The basic structure theory of $\QSP$ is established in \cite{Kol14}. The same proofs apply to the case of a pseudo-symmetric pair.

\begin{prop} \label{prop:QSP:properties}
\hfill
\begin{enumerate}\itemsep0.25cm
\item $\QSP$ is a right coideal subalgebra.
\item $\QSP$ is a maximal subspace of $\Uqg$ which specializes at $q=1$ to $U(\fkk)_s$.
\item $\QSP \cap \Uqh = U_q(\fkh^\phi)$.
\end{enumerate}
\end{prop}

We now explicitly describe a family of examples of QSPs built from a quasi-split symmetric pair. 

\begin{example} \label{example:pseudoquasisplit}
For any diagram involution $\tau$, the involution $\om \circ \tau$ is a pseudo-involution. 
Hence for any multiplicative character $\chi: \Qlat \to \C^\times$ one obtains the pseudo-involution $\phi = \Ad(\chi) \circ \om \circ \tau$ with $X = \emptyset$, explicitly given as
\begin{equation} \label{theta:example}
\phi(e_i) = -y_i^{-1} f_{\tau(i)}, \qq \phi(f_i) = -y_i e_{\tau(i)}, \qq \phi|_{\fkh} = -\tau|_{\fkh}
\end{equation}
where $y= (y_i)_{i \in I} \in (\C^\times)^{\IS}$ is such that $y_i = \chi(\al_{\tau(i)})$.
The corresponding pseudo-fixed-point subalgebra and its q-deformation are then given by
\begin{align}
\label{k:example}
\fkk &= \C\langle \{ f_i - y_i e_{\tau(i)} \}_{i \in \IS}, \, \{ h_i-h_{\tau(i)} \}_{i \in \IS, \, i \ne \tau(i)} \rangle,  \\[2pt]
\label{B:example}
\QSP &= \bsF\langle \{ B_i \}_{i \in \IS}, \, \{ \Kg{h_i-h_{\tau(i)}} \}_{i \in \IS, \, i \ne \tau(i)} \rangle,
\end{align}
where $B_i = F_i - \gamma_i q_i^{-\tfrac{a_{i \, \tau(i)}}{2}} E_{\tau(i)} \Kg{i}^{-1} + \sigma_i \Kg{i}^{-1}$.
Here $\sigma_i = 0$ if $\tau(i) \ne i$ or if there exists $j \in \IS$ such that $\tau(j)=j$ and $a_{ij}$ is odd, and each $\gamma_i \in \bsF^\times$ specializes to $y_i$.\hfill \rmkend
\end{example}

\subsection{Basic K-matrices} \label{sec:universalK:QSP}
Henceforth, we fix a QSP $(\Uqg,\QSP)$ and the associated algebra automorphism $\phi_q: \Uqg \to \Uqg$.
The QSP is naturally equipped with a cylindrical structure.
The following result deals with the intertwining identity \eqref{K:intw} and the coproduct identity \eqref{K:coproduct}, while the support property \eqref{K:support} is discussed in Section \ref{sec:tensorK:QSP}.
For $\la \in \Qp$, let $U_q(\fkn^+)_\la$  be the corresponding root space of the positive part $U_q(\fkn^+) \subset \Uqg$.

\begin{theorem}\label{thm:stdcylstr} \mbox{} 
\begin{enumerate}
\item
There is a unique $\Ups \in \End(\FF{}{})$ of the form
\eq{
\Ups = \sum_{\la \in \Qp} \Ups_\la, \qq \Ups_\la \in U_q(\fkn^+)_\la, \qq \Ups_0=1\,,
}
such that the linear relation \eqref{K:intw} holds with $K = \Ups$ and $\psi = \phi_q^{-1}$.
\item
The coproduct formula \eqref{K:coproduct} holds with $J = R_\phi$, where $R_\phi$ denotes the universal R-matrix of the Hopf subalgebra
 $U_q(\g_X) \cdot U_q(\fkh^\phi)$.\footnote{The element $R_\phi$ was introduced in \cite[Def.~6.7, Lem.~6.8]{AV22} as the R-matrix of $U_q(\g_X)$ with a Cartan correction. More precisely, $R_\phi=(T_\phi^{-1} \ot T_\phi^{-1}) \cdot \Delta(T_\phi)$, where $T_\phi$ is a modified quantum Weyl group operator for the longest element of $W_{X}$.}
\end{enumerate}
In particular, $(\phi_q^{-1}, R_\phi, \Ups)$ is a cylindrical structure on $(\Uqg,\QSP)$, called \emph{standard}.
\end{theorem}

As stated above, the theorem was proved in \cite{AV22}, building on the following earlier results.
H. Bao and W. Wang established (1) in \cite[Ch. 2]{BW18a} for the family of quasi-split quantum symmetric pairs with $\fkg = \fksl_N$ and $\tau$ nontrivial.
M. Balagovi\'{c} and S. Kolb generalized (1) to all quantum symmetric pairs in \cite[Sec.~6]{BK19}.\footnote{
In \cite{BW18a, BK19}, the QSP parameters were constrained by requiring an analogue of the bar involution on $\QSP$, which was directly involved in the defining equation  $\Ups$. In our approach, \eqref{K:intw} takes the form $\Ups \cdot b = \phi_q^{-1}(b) \cdot \Ups$ for any $b\in\Uqk$. Then, a linear relation generalizing the one from \cite{BW18a, BK19} appears by observing that $\phi_q^{-1}$ is trivial on $U_q(\g_X) \cdot U_q(\fkh^\phi)$ and acts on the generators $B_i$ as a composition of the usual bar involution on $\Uqg$ and an involution of the QSP parameters, see \cite[Prop.~8.3]{AV22}. Based on this observation, the existence of a QSP bar involution was later obtained in \cite{Kol22} from the existence of $\Ups$, as defined by \eqref{K:intw}, and then imposing the constraints.
}
The proof of (2) appeared in \cite[Sec. 9]{BK19} for finite types. The general case was proved in \cite{AV22} and relied on a crucial support property of $R \cdot R_\phi^{-1}$, see \cite[Prop.~4.3]{AV22}. 

Finally, in \cite{RV20} it was pointed out that the approach of \cite{BW18a,BK19} directly extends to quantum \emph{pseudo}-symmetric pairs of finite type and the results of \cite{AV22} are proved in this generality.\\

The proof is a generalization of Lusztig's approach to the universal R-matrix, see \cite{Lus94}. 
The key tool is given by skew derivations, certain linear maps on $U_q(\fkn^+)$ which interact nicely with the bialgebra structure of $\Uqg$, as well as the pairing between $U_q(\fkn^+)$ and $U_q(\fkn^-)$, yielding a recursive proof of the existence of each $\Ups_\la$.
An important step, the so-called \emph{fundamental lemma of quantum symmetric pairs}, first formulated in \cite[Conj.~2.7]{BK15}, was uniformly established in \cite{BW21} in the Kac-Moody setting.\\

In \cite{BK19}, gauge transformations are implicitly used to obtain a cylindrical structure for any QSP with $\dim(\fkg)<\infty$ and $\psi$ equal to a diagram automorphism.
In this case, the resulting universal K-matrix acts on $\Mod_{\sf fd}(\Uqg)$, yielding matrix solutions of a constant reflection equation. 
By gauge transformations in terms of an invertible $g \in \End(\FF{}{})$ such that $\Ad(g)$ preserves $\Uqg$, we obtain a family of related cylindrical structures.
A distinguished example is given by the \emph{semistandard} twist pair, which consists of the coalgebra anti-automorphism $\omega \circ \tau$ and the trivial Drinfeld twist $1 \ot 1$. 
It can be obtained by gauge-transforming by quantum Weyl group operators, see \cite[Sec.~8.10]{AV22}. Note that it was essentially already described in \cite[Sec.~7.2]{BK19}, up to conventions.

\subsection{Generalized parameters} \label{prob:generalizedconstraints}
The assumption $(\gamma,\sigma) \in \Gamma_q \times \Sigma_q$ is a natural condition in the proof of Proposition \ref{prop:QSP:properties} (3), see \cite[Sec.~5]{Kol14}, but it appears to be amenable to a generalization.
In \cite{BB10}, P. Baseilhac and S. Belliard introduced a family of $\Uqg$-comodule deformations of $U(\g^\om)_\si$, for $\g$ of affine type, called \emph{generalized q-Onsager algebras}.
These algebras can be embedded as coideal subalgebras in $\Uqg$. They are in fact examples of QSP algebras $\QSP$ as defined in Definition \ref{def:QSP}, but with parameters $(\gamma,\sigma)$ chosen from the set
\begin{align*}
\Gamma\Sigma_q &= \big\{ (\ga,\si) \in (\bsF^\times)^{I} \times \bsF^{I} \, \big| \, 
\forall i \in I \; (\al_i,\theta(\al_i)) = 0 \; \Rightarrow \; \ga_i = \ga_{\tau(i)} \qu \text{and} \\
& \hspace{40mm} \forall i \in I\backslash I_{\sf ns}\; \si_i = 0 \qu \text{and} \qu \forall i,j \in I_{\sf ns} \; P_{ij} \text{ holds} \}
\end{align*}
where $I_{\sf ns}=\{ i \in \IS \, | \, \theta(\al_i) = -\al_i \}$ and $P_{ij}$ stands for the condition\footnote{Omitting the third alternative from $P_{ij}$, we recover $\Gamma_q \times \Sigma_q$.}
\[
a_{ij} \in 2\Z \qq \text{or} \qq \si_j = 0 \qq \text{or} \qq q_i^{r-1} \ga_i = \Big( \tfrac{1+q_i^r}{1-q_i^2} \si_i \Big)^2 \text{ with } 0<r \le -a_{ij} \text{ odd}.
 \]
For example, if $\fkg = \fksl_N$ and $\phi = \omega$ then we can consider the coideal subalgebra generated by the elements
\begin{equation}
F_i - q^{-1} \ga_i E_i \Kg{i}^{-1} \pm (1-q_i) \ga_i^{1/2} \Kg{i}^{-1}. 
\end{equation}

The following conjecture is a refinement of one in \cite{RV20} and is motivated by the apparent equivalence (in many examples) of the conditions $\QSP \cap \Uqh = U_q(\fkh^\phi)$ and $(\ga,\si) \in \Gamma\Sigma_q$, and by the study of matrix solutions of reflection equations in finite-dimensional $U_q(\fkg)$-modules.

\begin{conjecture}
For every $(\ga,\si) \in \Gamma\Sigma_q$, the pair $(\Uqg,\QSP)$ admits a cylindrical structure.
Moreover, up to gauge transformation, every invertible universal solution of \eqref{K:RE} for $\Uqg$, with $\dim(\g) < \infty$, arises this way.
\end{conjecture}


\subsection{Explicit formulae for $\Upsilon$}\label{ss:factorization}
The construction of the element $\Ups$ in Theorem \ref{thm:stdcylstr} is not explicit.
If $\fkg = \fksl_2$ and $\phi=\omega$, so that $\QSP \subset U_q(\fksl_2)$ deforms $U(\fkso_2) \subset U(\fksl_2)$, one can find $\Ups$ by ad-hoc computations, which first appeared in \cite{KSS93} and \cite{tDHO98}. More generally, explicit computations have been carried out in \cite{DK19} for many 
cases of restricted rank one\footnote{We refer to the dimension of $(\fkh')^{-\phi}$ as the \emph{restricted rank} of $\phi$.} in finite types with $\phi^2=\id$.\\

An important consequence of the proof of Theorem \ref{thm:stdcylstr} is that $\Ups_\la = 0$ if $\phi(\la) \ne -\la$.
In other words, $\Ups$ is supported on the nonnegative part of the restricted root lattice. 
The latter is given by
\[
\Sp_\Z \{ \al_i|_{\fkh^{-\phi}} \}_{i \in \IS} = \Sp_\Z \{ \tfrac{1}{2}(\al_i - \phi(\al_i)) \}_{i \in \IS}
\]
Consider the \emph{restricted Weyl group} $\wt W$, i.e. the Weyl group of the restricted root lattice.
It has a presentation as a Coxeter group, see, \eg \cite[Sec.~4]{RV22} and references therein.
In \cite{DK19}, $\Ups$ was obtained as an ordered product of explicit K-matrices of restricted rank one. 
Combined with a result from \cite{WZh23}, this yields a factorization of $\Upsilon$ for all QSPs of finite type, which corresponds to a reduced expression of the longest element of $\wt W$, in analogy with the factorization of the quasi-R-matrix in finite types, see \cite{KR90,LS90}.\\

In contrast, there is no explicit expression of universal K-matrices for quantum affine algebras.

\begin{conjecture}\label{conj:K-factorization}
	Given a quantum symmetric pair of affine type, $\Ups$ admits a \emph{loop-triangular} factorization analogous to that of the (quasi-)R-matrix of quantum affine algebras, see \cite{KhT92, LSS93, Be94, Da98}. 
\end{conjecture}

For affine symmetric pairs of restricted rank one, the K-matrix is {\em abelian}, \ie supported only on imaginary root vectors. Hence, the conjecture is obviously true in this case. Note that, the resticted root system always contains all imaginary roots. Thus, the abelian factor is always expected to be non-trivial. We also expect that each factor is expressed by an infinite product of restricted rank one K-matrices, indexed by a convex order on the positive restricted root system.

\subsection{QSP weight modules}\label{ss:QSPweight}
The tensor K-matrix associated to $\Ups$ does not act on any tensor product of a $\QSP$-module and a $\Uqg$-module.
We need to specify a suitable category of weight $\QSP$-modules, 
see \cite{AV25b} and cf.~\cite{BW18b,Wa24}.
Set
\eq{
\Plat_\phi = \Plat / \Plat^{-\phi}
}
with the canonical projection $[\cdot]_\phi : \Plat \twoheadrightarrow \Plat_\phi$.
There exists a pairing $\Plat_\phi \times \Qvext \to \tfrac{1}{2}\Z$ defined by $\zeta(h) = \la(h + \phi(h))/2$ if $\zeta = [\lambda]_\phi \in \Plat_\phi$. 
The restriction to $\Plat_\phi \times (\Qvext)^\phi$ is nondegenerate and takes values in $\Z$.
Note that $U_q(\fkh^\phi) = \bsF \langle \{ \Kg{h} \}_{h \in (\Qvext)^\phi} \rangle$.\\

We call a $\QSP$-module $M$ a \emph{(type ${\bf 1}$) QSP weight module} if
\eq{
M = \bigoplus_{\zeta \in \Plat_\phi} M_\zeta, \qq M_\zeta = \{ m \in M \, | \, \forall h \in (\Qvext)^\phi \, \Kg{h} \cdot m = q^{\zeta(h)} m \}.
}
We denote the full subcategory of QSP weight modules by $\cW_\phi \subseteq \Mod(\QSP)$ and the corresponding completion of $\QSP$ by $\End(\FF{\phi}{}) = \End(\FF{\cW_\phi}{})$.\footnote{Note that $\cW_\phi = \Mod(\QSP)$ if and only if $\phi|_{\fkh} = \omega|_{\fkh}$.}
The category $\cW_\phi$ has many desirable properties, see \cite[Prop.~5.2.3]{AV25b}: 
\begin{enumerate}
\item up to twisting by algebra automorphisms of $\QSP$, all finite-dimensional irreducible $\QSP$-modules are objects in $\cW_\phi$;
\item by restriction to $\QSP$, every object in $\cW$ belongs to $\cW_\phi$;
\item $\cW_\phi$ is a right module category over the monoidal category $\cW$ (and hence over the subcategories $\cO$ and $\cO^{\sf int}$);
\item $\cW_\phi$ separates points in $\QSP$.
\end{enumerate}

\subsection{Tensor K-matrices for quantum symmetric pairs} \label{sec:tensorK:QSP}
Let $\TKM{}$ be the tensor K-matrix corresponding to $K=\Ups$, \ie $\TKM{} \coloneqq (R^\psi)_{21} \cdot \Ups_2 \cdot R$.

\begin{theorem}[\cite{AV25b}] \label{thm:support}
The tensor K-matrix $\TKM{}$ lies in $\End(\FF{\phi}{} \boxtimes \FF{}{})$.
\end{theorem}

The proof relies on the maximality property \eqref{B:maximal}, see \cite[Thm.~5.6.1]{AV25b}, which holds as a consequence of Proposition \ref{prop:QSP:properties} (2), a study of the classical limit of $\Ups$, and the Iwasawa decomposition \eqref{k:Iwasawa}.
For QSPs of finite type and finite-dimensional $\QSP$-modules, the support property is proved in \cite{Kol20}.

\begin{remark}
In contrast with Section~\ref{ss:factorization}, tensor K-matrices of finite types are not known to 
decompose according to the restricted root system. 
For QSPs of split and quasi-split affine types, new \emph{loop} generators have been introduced in \cite{LW21, Zh22, LWZ23, LWZ24}. 
No loop-triangular decomposition is known for their tensor K-matrices. 
However, for $U_q(\wh\fksl_N)$, strong evidence is given by the K-matrix presentations of the QSPs in \cite{MRS03, CGM14}.
On a related point, the authors of \cite{LBG25a} provide a twisted LDU decomposition of an algebra-valued matrix which should correspond to $\mathbb{K}_{\bullet,V}(z)$ in the current work, in the case where $\Uqk \subset U_q(\wh\fksl_2)$ is the q-Onsager algebra and $V$ is the two-dimensional irreducible $U_q(\wh\fksl_2)$-module.
\rmkend
\end{remark}

\subsection{The boundary transfer matrix for the dual K-matrix} \label{sec:boundarytransfermatrices:qgpfinite}
We now apply Theorem~\ref{thm:t:Grothendieck} and Proposition~\ref{prop:wtKbg} to the standard cylindrical structure of $(\Uqg, \QSP)$. 
As in Section \ref{ss:dualK}, set $\wt K=D\cdot K^{-1}$, where $D$ is defined by \eqref{D:def}, and consider 
\eq{
\gTKM{} = (1 \ot \wt K) \cdot \TKM{} \in \End(\FF{\phi}{} \boxtimes \FF{}{}).
} 
\begin{prop} \label{prop:t:Grothendieck:finitetype}
Let $(\Uqg,\QSP)$ be a QSP of finite type.
There is a ring homomorphism $\TM_\phi: \big[ \Mod_{\sf fd}(\Uqg) \big] \to \QSP$
given by $\TM_\phi^{(V)}\coloneqq \Tr_2 \big( \gTKM{\bullet,V} \big)$ where 
$V \in \Mod_{\sf fd}(\Uqg)$.
\end{prop}

\begin{proof}
It is enough to apply the proof of Theorem \ref{thm:t:Grothendieck}, although the algebraic manipulations need to be understood in $\End(\FF{}{})$. 
To see that $\gTKM{\bullet,V}\in\QSP\ten\End(V)$ (and hence $\TM_\phi^{(V)}\in\QSP$), it is convenient to factorize $\gTKM{}$.
Recall the factorization $R = \kappa \cdot \Xi$ from Section \ref{sec:uniR}. 
For a $\Z$-linear function $f: \Plat \to \Plat$, consider $\kappa^f = (\kappa^f_{V,W})_{V,W \in \cW} \in \End(\FF{\cW}{2})$ with $\kappa^f_{V,W}$ acting on $V_\la \ot W_\mu$ as multiplication by $q^{(f(\la),\mu)}$
(so $\kappa = \kappa^\id$).
Then
\eq{\label{K:factorization}
\gTKM{\bullet, V} = (1 \ot D_V) \cdot (1 \ot \Ups_V^{-1}) \cdot \kappa^\phi_{\bullet,V} \cdot (\Xi_{V^\psi,\bullet})_{21} \cdot (1 \ot \Ups_V) \cdot \kappa_{\bullet,V} \cdot \Xi_{\bullet,V}
}
Since $V \in \cO$, $\Ups_V\in\End(V)$ and $\Xi_{\bullet,V}\in\Uqg\ten\End(V)$. Moreover, since  $V$ is finite-dimensional, one shows that $(\Xi_{V^\psi,\bullet})_{21}\in\Uqg\ten\End(V)$, see \cite[Thm.~5.6.1]{AV25b}.
Therefore, $\gTKM{\bullet, V} \in \QSP \ten \End(V)$.
\end{proof}

The elements $\TM_\phi^{(V)}$ are rather constrained, see Remark \ref{rmk:transfermatrices} (3), as the following result makes clear.

\begin{lemma}\label{lem:transfermatrix:trivial}
The image of $\TM_{\phi}$ lies in $U_q(\g_X)U_q(\fkh^\phi)^{\Plat}$, where $U_q(\fkh^\phi)^{\Plat}$ is the weight lattice refinement of $U_q(\fkh^\phi)$.
\end{lemma}

\begin{proof}
For simplicity, we discuss only the case $X=\emptyset$.
Then all factors in \eqref{K:factorization} are upper triangular matrices in the standard basis of $V$. 
In the case of $(\Xi_{V^\psi,\bullet})_{21}$, this holds since $\psi$ exchanges the positive and negative parts of $\Uqg$.
Moreover, $\Ups_V$, $\Xi_{\bullet,V}$ and $(\Xi_{V^\psi,\bullet})_{21}$ are unitriangular, so that 
\begin{align}
	 \TM_\phi^{(V)} 
	 =\Tr_V ((1 \ot D_V) \cdot \kappa^{\id+\phi}_{\bullet,V})
=\sum_{\mu\in\Plat}\dim(V_\mu)k_{\mu+\phi(\mu)-2\rho}
\end{align}
where $k_{\lambda}|_{V_\mu}=q^{(\lambda,\mu)}$.
\end{proof}

In particular, when $X=\emptyset$ and $\tau=\id$, $\TM_\phi^{(V)}$ is a scalar multiple of the identity.

\subsection{The boundary transfer matrix for the Drinfeld element}
We now assume that $(\Uqg, \QSP)$ is a QSP of finite type. 
In this case, the standard cylindrical structure is gauge equivalent to $(\varphi, R_{21}^{-1}, K^{\sf BK})$ where $\varphi\colon\Uqg\to\Uqg$ is a quasitriangular Hopf algebra automorphism, and $K^{\sf BK}\in\End(\FF{}{})$ is the basic K-matrix considered in \cite[Sec.~7.3]{BK19}. 

\begin{prop}\label{prop:kolb}
Whenever $\varphi=\id$, there is a ring homomorphism $\TM_{\sf BK}$ from $\big[ \Mod_{\sf fd}(\Uqg) \big]$ to $\QSP$ given by
$\TM_{\sf BK}^{(V)}\coloneqq \Tr_2 \big( (1\ten u_V)\mathbb{K}^{\sf BK}_{\bullet, V} \big)$
where $V\in\Mod_{\sf fd}(\Uqg)$, $u$ is the Drinfeld element, and $\mathbb{K}^{\sf BK}$ is the tensor K-matrix associated to $K^{\sf BK}$.	
\end{prop}

\begin{proof}
	It is enough to apply Theorem~\ref{thm:t:Grothendieck} and Proposition~\ref{prop:gen-wtKbg}.
\end{proof}

\begin{remarks}\label{rmk:kolb-gendual}
\hfill
\begin{enumerate}\itemsep0.25cm
\item 	
\label{rmk:kolb}
The map $\TM_{\sf BK}$ coincides with the homomorphism $\bm\Phi : [\Oint] \to \Uqk$ constructed by Kolb in \cite{Kol20}, which is proved to be surjective over the centre of $\Uqk$.
Kolb's construction extends to the case $\varphi\neq\id$, provided that $\Oint$ is replaced by its equivariantization with respect to $\varphi$. The result above extends under the same condition.
\item 
\label{rmk:generalizeddualK:QSP}
By Proposition~\ref{prop:gen-wtKbg}, we obtain a boundary transfer matrix for any generalized dual K-matrix $\wt K=D\cdot (K')^{-1}$ where $K'$ is  a basic K-matrix with the same underlying twist pair of $K$.

From this point of view, it is convenient to consider the \emph{semistandard} cylindrical structure on $(\Uqg, \QSP)$, which has the form $(\omega\circ\tau, 1\ten 1, K)$ where $K\in\End(\FF{}{})$ is the semistandard K-matrix, see \cite{AV22}. Thus, we obtain a boundary transfer matrix map by setting $K'$ equal to the semistandard K-matrix of an arbitrary QSP subalgebra whose underlying diagram automorphism is $\tau$.
\rmkend
\end{enumerate}
\end{remarks}

\Omit{
\andreacomment{It would be interesting to understand if Stefan's map is somewhat canonical from our point of view (= gauge invariant in some sense). If not, are there similar maps $\TM$ for the standard or the semistandard cylindrical structure?}

\bartcomment{Stefan's $K'$ works nicely for the case $g = T_I \cdot T_X$. 
It is not invariant under all gauge transformations, since $u^{-1} \cdot D$ satisfies the same coproduct formula as $g \cdot \Ad(\gamma^{-1}) \cdot \Ups$, i.e. the coproduct formula with $\psi = \id$. 
My comment below Remark 3.5.2 essentially tells you how to choose $K'$ for another $g$ such that the resulting cylindrical structures are the same.
Are you proposing a weaker notion of invariance?}
}

\section{Boundary transfer matrices for quantum affine symmetric pairs} \label{sec:affine}

Henceforth, assume that our QSP $(\Uqg,\QSP)$ is of untwisted affine type.
We are actually interested in finite-dimensional modules of untwisted quantum loop algebras, viewing the ring homomorphisms $\TM_{\phi}$ as refined representation-theoretic tools, in the spirit of \cite{FR99}. 
We will define spectral analogues of $\TM^{(V)}_\phi$. 
The simplification explained in Lemma \ref{lem:transfermatrix:trivial} does not occur since the relevant matrices are not all triangular in the weight bases of these modules.

\subsection{Affine Lie algebras \cite{Kac90}}
Let $\fkg$ be of untwisted affine type.
We write 
\eq{
I = \{ 0, 1,\ldots, r\}, \qq \Ifin = \{ 1,\ldots, r\}
}
so that $\gfin = \C \langle \{ e_i,f_i \}_{i \in \Ifin} \rangle$ is the underlying finite-dimensional Lie algebra with Cartan subalgebra $\hfin = \bigoplus_{i \in \Ifin} \C h_i$ and weight lattice $\Pfin \subset \hfin^*$, given by the $\Z$-dual of $\Sp_\Z \{ h_i \}_{i \in \Ifin}$.

Let $c \in \fkh'$ be the canonical central element and $h^\vee = \rho(c)$ the dual Coxeter number. 
We have the linear decomposition $\fkh' \cong \hfin \oplus \C c$ with associated projection $\overline{\phantom{\imath} \hspace{-4pt} \cdot \hspace{-4pt} \phantom{\imath}} : \fkh' \to \hfin$.
Denote the monoidal category of finite-dimensional type-${\bf 1}$ $U_q(\fkg')$-modules by $\cC$.
For all $V \in \cC$ we have the decomposition
\eq{
V = \bigoplus_{\la \in \Pfin} V_\la, \qq V_\la = \{ v \in V \, | \, K_i \cdot v = q_i^{\la(\overline{h_i})v} v \text{ for all } i \in I \}.
}
Consider the untwisted loop algebra $\Lg = \gfin \otimes \C[t, t^{-1}]$ and its q-deformation $U_q(\Lg)$, which identifies with the quotient Hopf algebra $U_q(\fkg') / (K_\del - 1)$.
The category $\cC$ identifies with the monoidal category of finite-dimensional (type-${\bf 1}$) $U_q(\Lg)$-modules and let $\End(\FF{\cC}{})$ be the corresponding completion of $U_q(\Lg)$.

\subsection{Grading shift and spectral R-matrices}

For any nontrivial group homomorphism $s: \Qlat \to \Z$ such that $s(\al_i) \ge 0$ we let $\Sigma^{s}_z: U_q(\fkg) \to U_q(\fkg) \ot \bsF[z,z^{-1}]$ be the corresponding grading shift, \ie the Hopf algebra homomorphism acting on the root space $U_q(\fkg)_\la$ as multiplication by $z^{s(\la)}$.
For any $U_q(\fkg')$-module $V$, we denote the corresponding representation map by $\pi_V: U_q(\fkg') \to \End(V)$.
Then we consider the grading-shifted representation 
\eq{
\pi^s_{V,z} = \pi_V \circ \Sigma^s_z: U_q(\fkg') \to \End(V) \otimes \bsF[z,z^{-1}].
}

Following \cite{Dri86}, the coefficients of the formal power series
\eq{ 
R^s(z) = (\id \ot \Sigma^s_z)(R) = (\Sigma^s_{z^{-1}} \ot \id)(R)
}
have a well-defined action on $V \ot W$ for all $V,W \in \cC$. 
We denote the corresponding $\End(V \ot W)$-valued formal power series solution of the spectral Yang-Baxter equation by $R^s_{V,W}(z)$.

\subsection{Crossing symmetry}

We now derive a functional equation for the action of $R^s_{V,W}(z)$ on tensor products of finite-dimensional $U_q(\fkg')$-modules, following \cite[Sec.~5.2]{FR92}.
Such a relation is also called \emph{crossing symmetry} and is simply a spectral analogue of \eqref{R:doubledual2}.
By \eqref{R:antipode} we have
\eq{ \label{spectralR:antipode}
(\id \ot S^{-1})(R^s(z)) = R^s(z)^{-1}\,.
}
Hence, for arbitrary $V,W \in \cC$, by \eqref{spectralR:antipode} we obtain
$R^s_{V,{}^*W}(z) = (R^s_{V,W}(z)^{-1})^{\t_2}$
and therefore
\eq{ \label{spectralR:doubledual1}
\hspace{-4pt} 
R^s_{V,{}^{**}W}(z) = (((R^s_{V,W}(z)^{-1})^{\t_2})^{-1})^{\t_2}.
} 

The specific grading shift we need is the \emph{$\tau$-minimal grading shift} and is defined in terms of $\tau$, the diagram automorphism factor of the Lie algebra automorphism $\phi$ used to define $\QSP$.
Namely, we define $s \in \Hom_{\sf grp}(\Qlat,\Z)$ by 
\eq{
s(\al_i) = \begin{cases} 1 & \text{if } i \in \{ 0,\tau(0) \}, \\
0 & \text{otherwise}. \end{cases}
}
We will abbreviate
\eq{
\Sigma_z = \Sigma_z^{s}, \qq R(z) = R^{s}(z), \qq \pi_{V,z} = \pi^{s}_{V,z}.
}

On the other hand, let $\Sigma^{\sf hom} = \Sigma^{s_{\sf hom}}$ be the homogeneous grading shift, so $s_{\sf hom}(\al_i) = \del_{i0}$ (clearly, $s = s_{\sf hom}$ if and only if $\tau(0)=0$), and consider the corresponding spectral R-matrix $R^{\sf hom}(z) = (\id \ot \Sigma^{\sf hom}_z)(R)$.
It is convenient to have a conversion rule for these two grading shifts and the corresponding spectral R-matrices.
Set 
\eq{
f \coloneqq  s(\del) / s_{\sf hom
}(\del) = |\{ 0,\tau(0) \}| \in \{1,2\}, \qq \qq h^\vee_\phi \coloneqq  h^\vee/f \in \tfrac{1}{2}\Z.
}
Let $\Qfin \coloneqq  \sum_{i=1}^r \Z \al_i \subset \Pfin$ and extend the restriction $(s - f \, s_{\sf hom})|_{\Qfin}$ to a group homomorphism: $\Pfin \to \Q$ in an arbitrary way.
This extension will also be denoted $s - f \, s_{\sf hom}$ and takes values in $\tfrac{1}{m} \Z$ for some positive integer $m$.
Since $s - f \, s_{\sf hom}$ annihilates $\del$, we obtain the rule $\pi_{V,z} = \Ad(z^{s - f \, s_{\sf hom}}) \circ \pi^{s_{\sf hom}}_{V,z^{f}}$.
It implies
\eq{ \label{spectralR:conversion}  
\begin{aligned}
R_{V,W}(z) &= \Ad(\Id_V \ot z^{s - f \, {s_{\sf hom}}})( R^{\sf hom}_{V,W}(z^{f}) ) .
\end{aligned}
}

Denote the sum of fundamental $\gfin$-weights by $\rhofin$ and extend it to an element of $(\fkh')^*$ by setting $\rhofin(c)=0$.
Define linear maps $\Dfin_{V}: V \to V$ ($V \in \cC$) via
\eq{ \label{Daffine:def}
\Dfin_{V}(v) = q^{-2(\rhofin,\la) + 2(h^\vee_\phi s - h^\vee s_{\sf hom})(\la)} v, \qq v \in V_\la, \qq \la \in \Pfin
}
and collect them into the tuple $\Dfin = (\Dfin_{V})_{V \in \cC} \in \End(\FF{\cC}{})$.

\begin{lemma} \label{lem:spectralR:funcrel}
Set $p=q^{h_\phi^\vee}$.
The action of $R(z)$ on tensor products in $\cC$ satisfies
\eq{
\begin{aligned}
(((R_{V,W}(z)^{-1})^{\t_2})^{-1})^{\t_2} &= \Ad(1 \ot \Dfin^{-1}) \big( R(p^2 z) \big)_{V,W}.
\end{aligned}
}
\end{lemma}

\begin{proof}
We first review the proof in the case of the homogeneous grading.
By combining \eqref{spectralR:doubledual1} and $S^{-2} = \Ad(q^{2\rho})$ we obtain
\begin{equation}
(((R^{\sf hom}_{V,W}(z)^{-1})^{\t_2})^{-1})^{\t_2} = \big(\id \ot \Ad(q^{2\rho})\big) \big(R^{\sf hom}_{V,W}(z)\big).
\end{equation}
Next, we recall that $\rho = \rhofin + h^\vee \La_0$, where $\La_0 \in (\fkh')^*$ is defined by $\La_0(h_i) = \del_{i,0}$ for all $i \in I$, see e.g.~\cite[Ex.~7.16]{Kac90}. 
By a straightforward check on generators, one obtains $\Ad(q^{2\rho})  = \Ad(q^{2\rhofin}) \circ \Sigma^{\sf hom}_{q^{2h^\vee}}$.
Hence
\begin{equation}
(((R^{\sf hom}_{V,W}(z)^{-1})^{\t_2})^{-1})^{\t_2} = \big( \id \ot \Ad(q^{2\rhofin})\big) \big(R^{\sf hom}_{V,W}(q^{2h^\vee}z)\big),
\end{equation}
which is the desired identity for the homogeneous grading.
The identity for the $\tau$-minimal grading readily follows by replacing $z$ by $z^f$, conjugating by $\Id_V \ot z^{s - f \, s_{\sf hom}}$, and applying the conversion rules \eqref{spectralR:conversion}.
%
\end{proof}

\subsection{Spectral K-matrices} \label{sec:spectralK}

For any pseudo-involution $\phi: \fkg \to \fkg$, the automorphism $\phi_q$ descends to an automorphism of $U_q(\Lg)$ and the subalgebra $\QSP$ embeds into $U_q(\Lg)$ as a right coideal subalgebra.
We now follow \cite[Sec.~4]{AV25a} and observe that for the particular grading shift $\Sigma_z$ considered in Lemma \ref{lem:spectralR:funcrel}, there exists a cylindrical structure $(\psi,J,K)$ on the affine QSP $(\Uqg,\QSP)$ such that 
\eq{ \label{psi:Sigma}
\psi \circ \Sigma_z = \Sigma_{1/z} \circ \psi.
} 
As a consequence, one obtains the following relations for the \emph{spectral universal K-matrix} $K(z) = \Sigma_z(K) \in \End(\FF{\cC}{})((z))$ 
\begin{align}
K(z) \cdot \Sigma_z(b) &= \Sigma_{1/z}(b) \cdot K(z) \qq \text{for all } b \in \QSP, \label{K:spectral:intw} \\
\Delta_{z/y}(K(y)) &= J^{-1} \cdot (1 \ot K(z)) \cdot R(yz)^\psi \cdot (K(y) \ot 1), \label{K:spectral:coproduct}
\end{align}
where $\Delta_z \coloneqq  (\id \ot \Sigma_z) \circ \Delta$.
Furthermore the coefficients of $K(z)$ have a well-defined action on any $V \in \cC$, leading to $\End(V)$-valued formal Laurent series $K_V(z)$ satisfying a twisted $\QSP$-intertwining condition induced by \eqref{K:spectral:intw}.

Following the discussion in \cite[Secs.~5 and 6]{AV25a}, we note that this intertwining condition is a consistent finite linear system with a 1-dimensional solution space if $V \in \cC$ is irreducible as a $\Uqgp$-module. 
One obtains that $K_V(z)$ is proportional to a matrix $K^{\sf trig}_V(z)$ depending rationally on $z$, called \emph{trigonometric K-matrix}. 
Both $K_V(z)$ and $K^{\sf trig}_V(z)$ satisfy the generalized spectral reflection equation
\eq{ \label{KVz:RE}
\begin{aligned}
& R_{W^\psi,V^\psi}(z/y)_{21} \cdot K_W(z)_2 \cdot R_{V^\psi,W}(yz) \cdot K_V(y)_1 = \\
& \qq = K_V(y)_1 \cdot R_{W^\psi,V}(yz)_{21} \cdot K_W(z)_2 \cdot R_{V,W}(z/y)
\end{aligned}
}
for all $V,W \in \cC$, which first appeared in the works \cite{FM91,Che92,KS92}.\\

Taking $q \in \C^\times$ such that $1 \notin q^\Z$, in direct analogy with results on R-matrices \cite{EM02,FR92,KS95}, we propose the following conjecture.
\begin{conjecture}
For all $V \in \cC$ (not necessarily irreducible) and for generic QSP parameters, $K_V(z)$ is the Laurent series expansion of a meromorphic matrix-valued function. \hfill \rmkend 
\end{conjecture}

\subsection{K-matrices for Kirillov-Reshetikhin modules}
Recall that a Kirillov-Reshetikhin (KR) module is an irreducible object in $\cC$ with trivial Drinfeld polynomials except for one, say associated to node $i \in \Ifin$, where the roots are given by a geometric progression in $\bsF$ with ratio $q_i$ (see, \eg \cite{CP94, Her06} for more details). 
Suppose that $V, W \in \cC$ are KR modules and that $\phi$ is $\tau$-\emph{restrictable}, \ie the subdiagram $\Ifin$ of finite type is $\tau$-stable (or, equivalently, $\tau$ fizes the affine node).
Then, up to a gauge transformation, $K_V(z)$ and $K_W(z)$ satisfy a generalized spectral reflection equation replace \eqref{KVz:RE} where $\psi$ is replaced by a diagram automorphism, see \cite[Thm.~7.8.1]{AV25a}.

The study of solutions of spectral reflection equations in various special cases, see, \eg \cite{RV18,KOW22}, suggests the following conjecture.

\begin{conjecture}
Consider the equation \eqref{KVz:RE} with $\psi$ replaced by a diagram automorphism. 
Symmetrizable invertible solutions of this equation in tensor products of KR modules are rescaled actions of the grading-shifted universal K-matrix associated to a QSP of affine type with generalized parameter constraints as described in Section~\ref{prob:generalizedconstraints}.
\end{conjecture}

\begin{remark}
Conversely, by \cite{RV18} there are non-symmetrizable invertible solutions of \eqref{KVz:RE}, some of which cannot be obtained by taking limits of symmetrizable solutions. 
The general description of the corresponding coideal subalgebras $B \subset U_q(\fkg)$ is an open question. 
It appears that only the \emph{triangular q-Onsager algebra}, which embeds into $U_q(\wh\fksl_2)$, has been explicitly described, see \cite{BB13}. \rmkend
\end{remark}

\subsection{Spectral tensor and dual K-matrices}

Given the spectral version of the basic K-matrix from Section \ref{sec:spectralK}, the spectral tensor K-matrix is given by
\eq{
\TKM{}(z) \coloneqq  (\id \ot \Sigma_z)(\TKM{}) = \big( R(z)^\psi \big)_{21} \cdot (1 \ot K(z)) \cdot R(z),
}
in $\End(\FF{\phi}{} \boxtimes \FF{\cC}{})((z))$ with shifted coproduct formula
\eq{
(\id \ot \Delta_{z/y})(\TKM{}(y)) = J_{23}^{-1} \cdot \TKM{}(z)_{13} \cdot \big(R(y z)^\psi\big)_{23} \cdot \TKM{}(y)_{12}
}
in $\End(\FF{\phi}{} \boxtimes \FF{\cC}{} \boxtimes \FF{\cC}{})((y,z/y))$, see \cite[Sec.~6.4]{AV25b}.
We use the following invertible element of $\End(\FF{\cC}{})((z))$ as a \emph{parameter-dependent} gauge transformation: 
\eq{
\label{dualK:spectral:def} \wt K(z) \coloneqq  \Dfin \cdot K(pz)^{-1}.
}
Crucially, note the spectral parameter shift in terms of $p = q^{h_\phi^\vee}$.
Now we define
\eq{
\gTKM{}(z) \coloneqq  (1 \ot \wt K(z)) \cdot \TKM{}(z)  \qq \in \End(\FF{\phi}{} \boxtimes \FF{\cC}{})((z)).
}

\begin{lemma}  \label{lem:gTK:spectral}
The following spectral analogue of \eqref{gTK:coproduct2} holds:
\eq{ \label{gTK:spectral:coproduct2}
	(\id \ot \Delta_{z/y})(\gTKM{}(y)) = \wt R(yz)^{\gpsi(y)}_{23} \cdot \gTKM{}(z)_{13} \cdot R(yz)^{\gpsi(y)}_{23} \cdot \gTKM{}(y)_{12},
}
where 
\eq{
\gpsi(y) \coloneqq  \Ad(\wt K(y)) \circ \psi \qq \qq  \wt R(z) \coloneqq  \Ad(1 \ot \Dfin)\big( R(p^2 z)^{-1} \big) .
}
Moreover, for all $V,W \in \cC$ we have
\eq{ \label{bartildeR:spectral:identity}
(\wt R(yz)^{\gpsi(y)})_{V,W}^{\t_2} \cdot (R(yz)^{\gpsi(y)})_{V,W}^{\t_2} = \Id_{V \ot W}.
}
\end{lemma}

\begin{proof}
The coproduct formula \eqref{gTK:spectral:coproduct2} follows directly from \eqref{K:spectral:coproduct} and the following computation:
\begin{align}
\Delta_{z/y}(\wt K(y)) &= \Delta(\Dfin) \cdot \Delta_{z/y}\big(K(py)\big)^{-1} \\
&= (\Dfin \ot \Dfin) \cdot (K(py)^{-1} \ot 1) \cdot \big( R(p^2yz)^\psi \big)^{-1} \cdot (1 \ot K(pz)^{-1}) \cdot J \\
&= (\wt K(y) \ot 1) \cdot \wt R(yz)^\psi \cdot (1 \ot \wt K(z)) \cdot J, \label{dualK:spectral:coproduct}
\end{align}
where we relied on the $p$-shift appearing in the definition of $\wt K(z)$.
The identity \eqref{bartildeR:spectral:identity} is easily seen to be equivalent to 
\eq{
\wt R(yz)_{V^\psi, W}^{\t_2} \cdot R(yz)_{V^\psi,W}^{\t_2} = \Id_{V \ot W}.
}
which in turn is equivalent to Lemma \ref{lem:spectralR:funcrel}.
\end{proof}

\subsection{Boundary transfer matrices}
Let $[\cC]$ be the Grothendieck ring  of finite-dimensional (type ${\bf 1}$) $\Uqgp$-modules.

\begin{theorem} \label{thm:spectraltransfermatrix}
Let $(\Uqg,\QSP)$ be a QSP of affine type.
There is a ring homomorphism $\TM(z): [\cC] \to \End(\FF{\phi}{}) \ot \bsF((z))$ given by
$\TM^{(V)}(z) \coloneqq  \Tr_2 \gTKM{\bullet,V}(z)$ for $V \in \cC$.
In particular, for all $V,W \in \cC$, $[\TM^{(V)}(y),\TM^{(W)}(z)]=0$ in $\End(\FF{\phi}{}) \ot \bsF((y,z))$.
\end{theorem}

\begin{proof}
The proof is essentially the same as the proof of Theorem \ref{thm:t:Grothendieck}, with Lemma \ref{lem:gTK:spectral} substituting for \eqref{gTK:coproduct2} and \eqref{gJR-equation}. 
\end{proof}

\begin{remarks} \label{rmk:spectraltransfermatrix}
\hfill
\begin{enumerate}\itemsep0.25cm
\item 
The procedure outlined in Remark \ref{rmk:kolb-gendual} (2) also applies to $\TM(z)$.
Namely, we can replace $K$ in the definition \eqref{dualK:spectral:def} of the spectral dual K-matrix $\wt K(z)$ by the universal K-matrix $K'$ associated to a second QSP admitting the same twist pair $(\psi,J)$. For instance, one can consider two QSP algebras with the same underlying generalized Satake diagram, but with different QSP parameters.
\item Even if $K=K'$, in contrast with the case of finite type, $\TM(z)$ is typically highly nontrivial as we show below in the case of the augmented q-Onsager algebra.
\hfill \rmkend
\end{enumerate}
\end{remarks}

\subsection{Sklyanin's construction} \label{sec:quantumintegrability}
We explain how Sklyanin's formalism of commuting transfer matrices can be recovered as a special case.
Suppose $V ,W \in \cC$ satisfy $V^\psi=V$ and $W^\psi=W$ (in \cite[Thm.~7.8.1]{AV25a}, it is shown that if $\phi$ is $\tau$-restrictable, the longest element of the Weyl group of $\gfin$ sends $\alpha_i$ to $-\alpha_{\tau(i)}$ for all $i \in \Ifin$ and $V$ is a KR-module, then $V^\psi=V$). 
Then the \emph{untwisted} spectral reflection equation holds, as originally considered in \cite{Che84,Skl88}:
\eq{ \label{spectralK:untwRE}
\begin{aligned}
& R_{W,V}(z/y)_{21} \cdot K_W(z)_2 \cdot R_{V,W}(yz) \cdot K_V(y)_1 = \\
& \qq = K_V(y)_1 \cdot R_{W,V}(yz)_{21} \cdot K_W(z)_2 \cdot R_{V,W}(z/y).
\end{aligned}
}

For the spectral dual K-matrix defined by \eqref{dualK:spectral:def}, we obtain (either from \eqref{spectralK:untwRE} and the group-like nature of $\Dfin$, or from \eqref{dualK:spectral:coproduct} and the intertwining property of $R(z)$) the following dual reflection equation:
\eq{
\begin{aligned}
& R_{V,W}(z/y)^{-1} \cdot \wt K_W(z)_2 \cdot \wt R_{WV}(y z)_{21} \cdot \wt K_V(y)_1 = \\
& \qq = \wt K_V(y)_1 \cdot \wt R_{VW}(y z) \cdot \wt K_W(z)_2 \cdot R_{W,V}(z/y)^{-1}_{21}.
\end{aligned}
}
This equation appeared in \cite[(13)]{Skl88},\footnote{This is up to conventions: an overall left-right swap needs to be made, and strictly speaking one obtains a mild generalization, see, \eg \cite[(2.13)]{Vla15}.} with the interpretation of factorizable scattering at a second (in our case, right) boundary. 

Next, let $V_1, \ldots, V_N, W \in \cC$ be arbitrary.
In quantum integrability, the tensor product $V_1 \ot \cdots \ot V_N$ is interpreted as the space of states of a quantum spin chain and $W$ is called \emph{auxiliary space}.
By \cite[Lem.~6.3.1]{AV25b}, $V_1 \ot \cdots \ot V_N$ is an object in $\cW_\phi$.
Therefore $\TM^{(W)}(z) \in \End(\FF{\phi}{}) \ot \bsF((z))$ acts on $V_1 \ot \cdots \ot V_N$, yielding the linear map
\eq{ \label{boundaryTM:evaluated}
\TM^{(W)}_{V_1 \ot \cdots \ot V_N}(z) = \Tr_W \wt K_W(z) \cdot R_{W^\psi, V_1 \ot \cdots \ot V_N}(z)_{21} \cdot K_W(z) \cdot R_{V_1 \ot \cdots \ot V_N,W}(z)
}
where, as a consequence of \eqref{R:coproduct}, 
\begin{align}
R_{V_1 \ot \cdots \ot V_N,W}(z) &= R_{V_1,W}(z)_{1,N+1} \cdots R_{V_N,W}(z)_{N,N+1}, \\
R_{W^\psi, V_1 \ot \cdots \ot V_N}(z)_{21} &= R_{W^\psi,V_N}(z)_{N+1,N} \cdots R_{W^\psi,V_1}(z)_{N+1,1} .
\end{align}
This reproduces the factorized expression of the operator inside the trace over $W$, given by \cite[(20)-(24)]{Skl88} and, more generally, \cite[(2.12)]{Vla15}.\footnote{Again, this is up to conventions, and note that the precise identification can only be made for unitary trigonometric R-matrices (this is guaranteed if all $V_i$ are irreducible).}

\begin{example}
The special case under consideration in \cite[Secs.~4-6]{Skl88} is $\fkg = \wh\fksl_2$, $\phi = \Ad(\chi) \circ \om \circ \tau$ where $\tau$ is the nontrivial diagram automorphism, see Example \ref{example:pseudoquasisplit}. 
For a particular choice of $\chi$, $\QSP$ is generated by
\eq{
F_0 - q \xi^{-1} E_1 k_0^{-1}, \qq F_1 - q \xi E_0 k_1^{-1}, \qq k_0^{\pm 1} k_1^{\mp 1}
}
with $\xi \in \bsF^\times$ free.
We take $V\coloneqq W=V_1 = \ldots = V_N$ to be the irreducible 2-dimensional representation defined by
\eq{
E_1, F_0 \mapsto  \begin{pmatrix} 0 & 1 \\ 0 & 0 \end{pmatrix} , \qq
F_1, E_0 \mapsto  \begin{pmatrix} 0 & 0 \\ 1 & 0 \end{pmatrix} , \qq 
k_1, k_0^{-1} \mapsto  \begin{pmatrix} q & 0 \\ 0 & q^{-1} \end{pmatrix}.
}
We can now explicitly check that $V^{\psi}=V$ where $\psi$ is a gauge transformation of $\phi_q$ by a Cartan correction.
In this case, we have $\Dfin_V = \Id_V$. 
By solving the $\psi$-twisted $\QSP$-intertwining condition, we readily obtain a well-known trigonometric K-matrix:
\[
K^{\sf trig}(z) =  \begin{pmatrix} 1 & 0 \\ 0 & \frac{\xi-z^2}{\xi z^2 -1} \end{pmatrix} .
\]
Setting $z=1$ in $\frac{{\sf d}}{{\sf d}z}\TM^{(V)}_{V^{\ot N}}(z)$, one obtains the Hamiltonian for the XXZ spin chain with diagonal boundary conditions (with the two sets of QSP parameters related to each other).
This is a highly nontrivial operator, making clear that $\TM^{(V)}(z)$ is not a scalar matrix. \hfill \rmkend
\end{example}


\subsection{Connections} \label{sec:reptheoryapplications}

We discuss some, generally speculative, connections of the boundary transfer matrix with various representation-theoretic constructions.

\begin{enumerate} \itemsep1mm
	\item
	It is natural to interpret the image of $\TM(z)$ in $\QSP \otimes\Lfml{\C}{z}$ as a boundary analogue of a Bethe subalgebra.
	For example, in the case of (an extension of) the q-Onsager algebra, the expansion coefficients have been computed in \cite{BK05,BB13,Lem23}.
	\item
	Fusion formulas for the boundary transfer matrices $\TM^{(W)}_V(z)$ (TT relations), have been considered in special cases, originally by L. Mezincescu and R. Nepomechie in \cite{MN92}. 
	Using universal transfer matrices with auxiliary space equal to a spin-$j$ module of $U_q(\wh\fksl_2)$, this has been accomplished by G. Lemarthe in \cite[Thm.~4.2.5]{Lem23}, also see \cite{LBG25b}, for the alternating central extension $\mathcal{A}_q$ of the q-Onsager algebra, subject to conjectural identification of an explicitly constructed fused K-operator and the action of a tentative universal tensor K-matrix for $\mathcal{A}_q$.
	It would be nice to connect the universal element $\gTKM{}$ to the desired universal tensor K-matrix for $\mathcal{A}_q$, and approach boundary TT relations for other types from the universal transfer matrices discussed in this paper.
\item
To upgrade the map $\TM(z)$ to a q-character map along the lines of \cite{FR99,FM01} it is necessary to define a QSP analogue of the \emph{loop} Harish-Chandra map. This requires a Drinfeld-type loop presentation of $\Uqk$, which was recently obtained for various quasi-split types in the works \cite{LW21,Zh22,LWZ23,LWZ24}.
\item
We obtained a ring homomorphism: $[\cC] \to \QSP \ot\Lfml{\C}{z}$ by letting the second leg of $\gTKM{}$ act on a quantum group module.
One can instead evaluate the \emph{first} leg of $\gTKM{}$, yielding a group homomorphism $[\Mod_{\sf fd}(\QSP)] \to \Uqgp \ot\Lfml{\C}{z}$.
When $X$ does not contain the affine node, we expect it to generalize the boundary q-character map defined in \cite{Prz23,LP24} in quasi-split types, after composition with the {loop} Harish-Chandra map.
\item
Consider the evaluated transfer matrix $\TM^{(W)}_{V_1 \ot \cdots \ot V_N}(z)$, see \eqref{boundaryTM:evaluated}, in the case that $V_i = W(a_i)$ with $a_i \in \bsF^\times$ generic, for all $1 \le i \le N$. 
By \cite[Thm.~3.13]{Vla15}, the $N$ ``interpolated'' operators $\TM^{(W)}_{V_1 \ot \cdots \ot V_N}(a_i)$ can be identified with scattering matrices for systems with two boundaries.\footnote{The identification with $\TM_V^{(W)}(a_i)$ relies on the condition $R_{W,W}(1) = {\sf flip}_{W,W}$, known to hold for KR modules $W$, see also the discussion in \cite[App.~10]{FHR22}. }
They are q-connection matrices of two-boundary quantum Knizhnik-Zamolodchikov equations, see \cite[Sec.~5]{Che92}, but with trivial step parameter.

In affine type, the quantum group $\Uqg$ has two important categories of representations: $\Oint$ and $\cC$. 
They are related by intertwiners satisfying quantum KZ equations of type A, described in \cite{FR92} in terms of its universal R-matrix $R$. 
We expect that, for a suitable choice of dual K-matrix, the tensor K-matrix $\gTKM{}$ of an affine QSP can be used to define analogous intertwiners connecting certain infinite-dimensional $\Uqk$-modules with objects in $\cC$.

%
%
\item Finally, as mentioned in the introduction to \cite{AV25b}, an overarching goal that connects with other points mentioned here, is to formulate a universal description of spectra for quantum integrable systems with reflecting boundaries.
In the case of systems on a circle (periodic boundary conditions), the central idea is that the spectra are determined by so-called Baxter relations, which should be viewed as relations in the Grothendieck ring of a suitable category of $U_q(\fkb^+)$-modules, see \cite{HJ12,FH15}.
It is natural to suspect a similar theory can be developed for systems with reflecting boundary conditions.

We can make the following convenient observation, which means that $\TM^{(W)}(z)$ can be defined for suitable $U_q(\fkb^+)$-modules for certain affine QSPs. 
Suppose that the dual K-matrix $\wt K$ appearing in $\gTKM{}$ is itself obtained, see Remark \ref{rmk:spectraltransfermatrix}, from the universal basic K-matrix $K'$ associated to an affine quantum symmetric pair.
Further, suppose that the affine quantum symmetric pair defining the basic K-matrix $K$ is quasi-split.
If both $K$ and $K'$ are part of standard cylindrical structures, then indeed $\gTKM{}$ lies in a completion of $\Uqk \ot U_q(\fkb^+)$.

\Omit{
\bartcomment{Made it a little more explicit. Do we want go further and claim that $\gTKM{}$ acts on modules in the HJ-category $\cO_{HJ}$ of $U_q(\fkb^+)$-modules, or leave it for the next paper?
For this we need a statement that in the quasi-split case, the standard cylindrical structure satisfies 
\[
\TKM{} = (R^{\phi_q^{-1}})_{21} \cdot (1 \ot \Upsilon) \cdot R =\sum_{\la \ge 0} b_\la \ot x_\la
\]
with $b_\la \in \Uqk$, $x_\la \in (U_q^{\ge 0})_\la$. 
Do we know this from results in paper 3?
Note that such a statement is necessary to conclude that $\gTKM{}$ acts on modules in the HJ-category $\cO_{HJ}$ of $U_q(\fkb^+)$-modules.
Indeed, along the lines of our argument in paper 2, we need to argue that for a suitable grading shift $s: \Qlat \to \Z$, the coefficients $\sum_{\la \ge 0, s(\la) = m} b_\la \ot x_\la$ appearing in 
\[
(\id \ot \Sigma_z)(\TKM{}) = \sum_{m \ge 0} z^m \sum_{\la \ge 0 \atop s(\la) = m} b_\la \ot x_\la
\]
give rise to a well-defined element of $\Uqk \ot \End(V)$ for any $V \in \cO_{HJ}$.
}
}
\end{enumerate} 


\providecommand{\bysame}{\leavevmode\hbox to3em{\hrulefill}\thinspace}
\providecommand{\MR}{\relax\ifhmode\unskip\space\fi MR }
\providecommand{\MRhref}[2]{
	\href{http://www.ams.org/mathscinet-getitem?mr=#1}{#2}
}
\providecommand{\href}[2]{#2}


\begin{thebibliography}{DCNTY19}
	
%
\bibitem[ATL19]{ATL19}
	A. Appel, V. Toledano~Laredo, 
	{\it Coxeter categories and quantum groups},
	Selecta Math. (N.S.) \textbf{25}, no.~3 (2019), p97. \MR{3984102}
	
\bibitem[ATL24]{ATL24}
	A. Appel, V. Toledano~Laredo,
	{\it Pure braid group actions on category O modules},
	Pure Appl. Math. Q. \textbf{20}, no.~1 (2024). 
	
\bibitem[AV22]{AV22}
	A. Appel, B. Vlaar, 
	{\it Universal {K}-matrices for quantum {K}ac--{M}oody algebras},
	Represent. Theory {\bf 26} (2022), 764--824.

\bibitem[AV25a]{AV25a}
A. Appel, B. Vlaar, 
{\it Trigonometric {K}-matrices for finite-dimensional representations of quantum affine algebras},
J. Eur. Math. Soc. (2025), published online first.

\bibitem[AV25b]{AV25b}
	A. Appel, B. Vlaar, 
	{\it Tensor K-matrices for quantum symmetric pairs},
	Commun. Math. Phys., {\bf 406}, no.~5 (2025), 100.



\bibitem[Ba82]{Ba82}
	R. Baxter,
	{\it Exactly solved models in statistical mechanics},
	Acad. Press Inc., 1989. Reprint of the 1982 original.

\bibitem[BB10]{BB10}
	P. Baseilhac, S. Belliard,
	{\it Generalized q-Onsager algebras and boundary affine Toda field theories}.
	Lett. Math. Phys. {\bf 93} (2010): 213--228.
	{\tt arXiv:0906.1215}.
	
\bibitem[BB13]{BB13}
	P. Baseilhac, S. Belliard, 
	{\it The half-infinite XXZ chain in Onsager's approach},
	Nucl. Phys. B {\bf 873}, no.~3 (2013), 550--584.
\bibitem[Be94]{Be94}
J. Beck, {\it Convex bases of PBW type for quantum affine algebras}, Comm. Math. Phys. {\bf 165} (1994), no.~1, 193--199.

\bibitem[Bic01]{Bic01}
	J. Bichon,
	{\it Cosovereign Hopf algebras},
	J. Pure Appl. Alg. {\bf 157}, no.~2-3 (2001), 121--133.

\bibitem[BK05]{BK05}
	P. Baseilhac, K. Koizumi, 
	{\it A deformed analogue of Onsager's symmetry in the XXZ open spin chain},
	J. Stat. Mech. \textbf{0510} (2005), P005.

\bibitem[BK15]{BK15}
	M. Balagovi\'{c}, S. Kolb, 
	{\it The bar involution for quantum symmetric pairs},
	Represent. Theory {\bf 19}, no.~8 (2015), 186--210.

\bibitem[BK19]{BK19}
	M. Balagovi\'{c}, S. Kolb, 
	{\it Universal {K}-matrix for quantum symmetric pairs},
	J. Reine Angew. Math. \textbf{747} (2019), 299--353. \MR{3905136}
	
%
%
\bibitem[Bro13]{Bro13}
	A. Brochier, 
	{\it  Cyclotomic associators and finite type invariants for tangles in the solid torus},
	Algebr. Geom. Topol. \textbf{13} (2013), 3365-3409.


%

\bibitem[BW18a]{BW18a}
H. Bao, W. Wang, 
{\it A new approach to {K}azhdan-{L}usztig theory of type {$B$} via quantum symmetric pairs},
Ast\'{e}risque, no.~402 (2018), vii+134.\MR{3864017}

\bibitem[BW18b]{BW18b}
	H. Bao, W. Wang, 
	{\it Canonical bases arising from quantum symmetric pairs},
	Invent. Math., no.~213 (2018), 1099--1177.

\bibitem[BW21]{BW21}
	H. Bao, W. Wang,
	{\it Canonical bases arising from quantum symmetric pairs of Kac-Moody type},
	Comp. Math. {\bf 157}, no.~7 (2021), 1507--1537.

\bibitem[BZBJ18]{BZBJ18}
	D. Ben-Zvi, A. Brochier, D. Jordan, 
	{\it  Quantum character varieties and braided module categories},
	Selecta Math. (N.S.) \textbf{24}, no. 5 (2018), 4711-4748.

\bibitem[CGM14]{CGM14}
	H. Chen, N. Guay, X. Ma, 
	{\it Twisted Yangians, twisted quantum loop algebras and affine Hecke algebras of type BC},
	Transactions of the American Mathematical Society, 366 (2014), 2517--2574.


\bibitem[Che84]{Che84}
	I. Cherednik, 
	{\it Factorizing particles on a half line, and root systems},
	Teoret. Mat. Fiz. \textbf{61}, no.~1 (1984), 35--44. \MR{774205}

%
\bibitem[Che92]{Che92}
	I. Cherednik, 
	{\it Quantum {K}nizhnik-{Z}amolodchikov equations and affine root systems},
	Comm. Math. Phys. \textbf{150}, no.~1 (1992), 109--136. \MR{1188499}

%

\bibitem[CP94]{CP94}
	V. Chari, A. Pressley, 
	{\it Quantum affine algebras and their representations},
	Representations of groups (Banff, AB, 1994), CMS Conference Proceedings, vol.~16, Amer. Math. Soc., Providence, RI, 1994, 59--78.


%

\bibitem[Dam98]{Da98}
I. Damiani, {\it La $R$-matrice pour les alg\`ebres quantiques de type affine non tordu}, Ann. Sci. \'Ecole Norm. Sup. (4) {\bf 31} (1998), no.~4, 493--523.

\bibitem[Dav10]{Dav07}
	A. Davydov,
	{\it Twisted automorphisms of Hopf algebras},
	Noncommutative structures in Mathematics and Physics, 
	Koninklijke Vlaamse Academie van Belgi\"{e} voor Wetenschappen en Kunsten, 2010, pp. 103--130.


\bibitem[DCK90]{DCK90}
C. De Concini, V.G. Kac, 
{\it Representations of quantum groups at roots of 1},
Operator algebras, unitary representations, enveloping algebras and invariant theory, Birkh\"{a}user (1990), 471--506.

\bibitem[DCNTY19]{DCNTY19}
K. De Commer, S. Neshveyev, L. Tuset, M. Yamashita, 
{\it Ribbon braided module categories, quantum symmetric pairs and Knizhnik–Zamolodchikov equations},
Comm. Math. Phys. {\bf 367} (2019), 717--769.

%
\bibitem[DK19]{DK19}
	L. Dobson, S. Kolb, 
	{\it  Factorisation of quasi $K$-matrices for quantum symmetric pairs},
	Selecta Math. (N.S.) {\bf 25}, no.~4 (2019), Paper No. 63, 55 pp.

\bibitem[DKM03]{DKM03}
	J. Donin, P. Kulish, A. Mudrov,
	{\it On a universal solution to the reflection equation},
	Lett. Math. Phys. {\bf 63} (2003), 179--194.

%
\bibitem[DN98]{DN98}
	M. Dijkhuizen, M. Noumi,
	{\it A family of quantum projective spaces and related q-hypergeometric orthogonal polynomials},
	Trans. Amer. Math. Soc. {\bf 350}, no. 8 (1998), 3269--3296.

\bibitem[Dri86]{Dri86}
	V. Drinfeld, 
	{\it Quantum groups},
	Proceedings of the ICM, Vol. 1, 2 (Berkeley, Calif., 1986), Amer. Math. Soc., Providence, RI, 1987, 798--820. \MR{934283}

\bibitem[Dri90]{Dri90}
V. Drinfeld, {\it Almost cocommutative Hopf algebras}, Leningrad Math. J. {\bf 1} (1990), no.~2, 321--342; translated from Algebra i Analiz {\bf 1} (1989), no.~2, 30--46.


%

\bibitem[EM02]{EM02}
	P. Etingof, A. Moura, 
	{\it On the quantum {K}azhdan-{L}usztig functor}, 
	Math. Res. Lett. \textbf{9} (2002), no.~4, 449--463. \MR{1928865}

\bibitem[Enr07]{Enr07}
	B. Enriquez, 
	{\it  Quasi-reflection algebras and cyclotomic associators},
	Selecta Math. (N.S.) {\bf 13}, no.~3 (2007), 391--463.


%
\bibitem[FH15]{FH15}
	E. Frenkel, D. Hernandez,
	{\it Baxter's relations and spectra of quantum integrable models},
	Duke Math. J. {\bf 164} (2015), 2407--2460.

\bibitem[FHR22]{FHR22}
	E. Frenkel, D. Hernandez, N. Reshetikhin,
	{\it Folded quantum integrable models and deformed W-algebras},
	Lett. Math. Phys. {\bf 112}, no.~4 (2022), 80.

\bibitem[FM01]{FM01}	
	E. Frenkel, E. Mukhin, 
	{\it Combinatorics of q-characters of finite-dimensional representations of quantum affine algebras},
	Comm. Math. Phys. {\bf 216} (2001), 23--57.

\bibitem[FM91]{FM91}
	L. Freidel, J. Maillet, 
	{\it Quadratic algebras and integrable systems}, 
	Physics Letters B \textbf{262} (1991), no.~2, 278--284.

\bibitem[FR92]{FR92}
	I. Frenkel, N. Reshetikhin, 
	{\it Quantum affine algebras and holonomic difference equations},
	Comm. Math. Phys. \textbf{146}, no.~1 (1992), 1--60. \MR{1163666}

\bibitem[FR99]{FR99}
	E. Frenkel, N. Reshetikhin, 
	{\it The q-characters of representations of quantum affine algebras and deformations of W-algebras},
	In Recent Developments in Quantum Affine Algebras and Related Topics, N. Jing and K. Misra (eds.), 
	Contemporary Mathematics {\bf 248} (1999), 163--205.

\bibitem[FRT90]{FRT90}
	L. Faddeev, N. Reshetikhin, L. Takhtajan, 
	{\it Quantization of Lie groups and Lie algebras},
	Leningrad Math. J. {\bf 1} (1990), 193.

\bibitem[FSHY97]{FSHY97}
	H. Fan, K.-J. Shi, B.-Y. Hou, Z.-X. Yang,
	{\it Integrable boundary conditions associated with the $Z_n \times Z_n$ Belavin model and solutions of the reflection equation},
	Int. J. Mod. Phys. A {\bf 12}, no.~16, (1997), 2809--2823.

%
\bibitem[GIK98]{GIK98}
	A. Gavrilik, N. Iorgov, A. Klimyk,
	{\it Nonstandard Deformation $U_q'(\fkso_n)$: The Imbedding $U_q'(\fkso_n) \subset U_q(\fksl_n)$ and Representations.}
	Symmetries in Science X (1998), 121--133.	

\bibitem[GK91]{GK91}
	A. Gavrilik, A. Klimyk, 
	{\it q-Deformed orthogonal and pseudo-orthogonal algebras and their representations},
	Lett. Math. Phys. \textbf{21} (1991), 215--220.

%
%
\bibitem[Her06]{Her06}
	D. Hernandez, 
	{\it The Kirillov-Reshetikhin conjecture and solutions of $T$-systems}, 
	J. Reine Angew. Math. {\bf 596} (2006), 63--87. \MR{2254805}
	
\bibitem[HJ12]{HJ12}
	D. Hernandez, M. Jimbo,
	{\it Asymptotic representations and Drinfeld rational fractions},
	Compos. Math. \textbf{148}, no.~5 (2012), 1593--1623.

%
%
%
\bibitem[Jim85]{Jim85}
	M. Jimbo, 
	{\it A q-analogue of $U(\fkg)$ and the Yang-Baxter equation},
	Lett. Math. Phys. \textbf{11} (1985), 63--69.

%

\bibitem[Kac90]{Kac90}
	V. Kac, 
	{\it Infinite-dimensional Lie algebras},
	3rd ed., Cambridge University Press (1990).

\bibitem[KhT92]{KhT92}
	S. Khoroshkin, V. Tolstoy,
	{\it The universal R-matrix for quantum untwisted affine Lie algebras}.
	Funct. Anal. Appl. {\bf 26} (1992), 69--71.
	
\bibitem[Kol14]{Kol14}
	S. Kolb, 
	{\it Quantum symmetric {K}ac-{M}oody pairs},
	Adv. Math. \textbf{267} (2014), 395--469. \MR{3269184}
	
\bibitem[Kol20]{Kol20}
	S. Kolb, 
	{\it  Braided module categories via quantum symmetric pairs},
	Proc. Lond. Math. Soc. (3) {\bf 121}, no.~1 (2020), 1--31. \MR{4048733}
	
\bibitem[Kol22]{Kol22}
	S. Kolb, 
	{\it The bar involution for quantum symmetric pairs -- hidden in plain sight},
	Hypergeometry, integrability and Lie theory, 69--77, Contemp. Math. {\bf 780}, Amer. Math. Soc. (2022).

\bibitem[Koo93]{Koo93}
	T. Koornwinder,
	{\it Askey-Wilson polynomials as zonal spherical functions on the $SU(2)$ quantum group}, 
	SIAM J. Math. Anal. {\bf 24}, no.~3 (1993), 795--813.

\bibitem[KOW22]{KOW22}
	H. Kusano, M. Okado, H. Watanabe, 
	\emph{Kirillov-Reshetikhin modules and quantum K-matrices}, Comm. Math. Phys. {\bf 405} (2024), no.~4, Paper No. 88, 27 pp.

\bibitem[KR90]{KR90}
	A. N. Kirillov\ and\ N. Reshetikhin, 
	{\it $q$-Weyl group and a multiplicative formula for universal $R$-matrices}.
	Comm. Math. Phys. {\bf 134}, no.~2 (1990), 421--431.

%

\bibitem[KS09]{KS09}
	S. Kolb, J. Stokman,
	{\it Reflection equation algebras, coideal subalgebras, and their centres},
	Sel. Math. {\bf 15}, no.~4 (2009), 621-664.

\bibitem[KS92]{KS92}
	P. Kulish, E. Sklyanin, 
	{\it Algebraic structures related to reflection equations},
	J. Phys. A: Math. Gen. \textbf{25}, no.~22 (1992), 5963.
	
\bibitem[KS95]{KS95}
	D. Kazhdan, Y. Soibelman, 
	{\it Representations of quantum affine algebras},
	Selecta Math. (N.S.) \textbf{1}, no.~3 (1995), 537--595. \MR{1366624}
	
\bibitem[KSS93]{KSS93}
P. Kulish, R. Sasaki, C. Schwiebert,
{\it Constant solutions of reflection equations and quantum groups},
J. Math. Phys. {\bf 34}, no.~1 (1993), 286--304.

\bibitem[KT09]{KT09}
	J. Kamnitzer, P. Tingley,
	{\it The crystal commutor and Drinfeld’s unitarized R-matrix}.
	Journal of Algebraic Combinatorics {\bf 29}, no.~3 (2009), 315--335.

\bibitem[KW92]{KW92}
	V. Kac, S. Wang, 
	{\it On automorphisms of {K}ac-{M}oody algebras and groups},
	Adv. Math. \textbf{92}, no.~2 (1992), 129--195. \MR{1155464}

\bibitem[KY20]{KY20}
	S. Kolb, M. Yakimov, 
 	{\it Symmetric pairs for Nichols algebras of diagonal type via star products},
	Adv. Math. \textbf{365} (2020), 107042.

%

%
%
\bibitem[LBG23]{LBG23}
	G. Lemarthe, P. Baseilhac, A.M. Gainutdinov, 
	{\it Fused K-operators and the $ q $-Onsager algebra},
	preprint at \href{https://arxiv.org/abs/2301.00781v2}{\tt arXiv:2301.00781} (2023).

\bibitem[LBG25a]{LBG25a}
	G. Lemarthe, P. Baseilhac, A.M. Gainutdinov, 
	{\it On the R-matrix presentation of the q-Onsager algebra},
	forthcoming (2025), announced by P. Baseilhac at the conference ``Algebraic structures in the Yang-Baxter equation'' at Heriot-Watt University, June 2024.

\bibitem[LBG25b]{LBG25b}
	G. Lemarthe, P. Baseilhac, A.M. Gainutdinov, 
	{\it TT-relations and the $ q $-Onsager algebra},
	forthcoming (2025).

\bibitem[Lem23]{Lem23}
G. Lemarthe,
{\it Universal solutions of the reflection equation, the q-Onsager algebra and applications},
PhD thesis, Universit\'{e} de Tours, \href{https://theses.hal.science/tel-04601310v1}{\tt theses.hal.science/tel-04601310} (2023).

\bibitem[Let97]{Let97}
	G. Letzter,
	{\it Subalgebras which appear in quantum Iwasawa decompositions},
	Canadian Journal of Math. {\bf 49}, no. 6 (1997), 1206--1223.

\bibitem[Let99]{Let99}
	G. Letzter,
	{\it Symmetric pairs for quantized enveloping algebras},
	J. Algebra {\bf 220}, no. 2 (1999), 729--767.

\bibitem[Let02]{Let02}
	G. Letzter, 
	{\it Coideal subalgebras and quantum symmetric pairs},
	New directions in Hopf algebras, Math. Sci. Res. Inst. Publ. {\bf 43}, Cambridge Univ. Press (2002), 117--165. \MR{1913438}
	
\bibitem[LP24]{LP24}
J.-R. Li, T. Prze\'{z}dziecki, \emph{Compatibility of Drinfeld presentations for split affine Kac-Moody quantum symmetric pairs}, Lett. Math. Phys. (2025). In press.

\bibitem[LS90]{LS90}
	S. Levendorskii, Y. Soibelman, 
	{\it Some applications of the quantum Weyl groups}, J. Geom. Phys. {\bf 7}, no.~2 (1990), 241--254. 

\bibitem[LSS93]{LSS93}
	S. Levendorskii, Y. Soibelman, V. Stukopin,
	{\it The Quantum Weyl Group and the Universal Quantum $R$-Matrix for Affine Lie Algebra $A_1^{(1)}$}, Lett. Math. Phys. {\bf 27} (1993), 253--264. 

\bibitem[Lus94]{Lus94}
	G. Lusztig, 
	{\it Introduction to quantum groups},
	Modern Birkh\"{a}user Classics, Birkh\"{a}user/Springer, New York, 2010, Reprint of the 1994 edition. \MR{2759715}

\bibitem[LW21]{LW21}
	M. Lu, W. Wang,
	{\it A Drinfeld type presentation of affine $\imath$quantum groups I: Split ADE type},
	Adv. Math. {\bf 393} (2021), 108111.

\bibitem[LWZ23]{LWZ23}
	M. Lu, W. Wang, W. Zhang,
	{\it Braid group action and quasi-split affine $\imath$quantum groups I},
	Represent. Theory {\bf 27} (2023), 1000--1040.

\bibitem[LWZ24]{LWZ24}
	M. Lu, W. Wang, W. Zhang,
	{\it Braid group action and quasi-split affine $\imath$quantum groups II}: higher rank, 
	Comm. Math. Phys. {\bf 405} (2024), no.~6, Paper No. 142, 33 pp.

%

\bibitem[MN91]{MN91}
	L. Mezincescu, R.I. Nepomechie, 
	{\it Integrable open spin chains with nonsymmetric R-matrices},
	J. Phys. A: Math. Gen. {\bf 24} (1991), L17.

\bibitem[MN92]{MN92}
	L. Mezincescu, R.I. Nepomechie, 
	{\it Fusion procedure for open chains},
	J. Phys. A: Math. Gen. {\bf 25} (1992), 2533.

%
\bibitem[MRS03]{MRS03}	
	A. Molev, E. Ragoucy, P. Sorba,
	{\it Coideal subalgebras in quantum affine algebras},
	Rev. Math. Phys. {\bf 15}, no. 8 (2003), 789--822.
	
\bibitem[Nou96]{Nou96}
	M. Noumi, 
	{\it Macdonald’s symmetric polynomials as zonal spherical functions on some quantum homogeneous spaces},
	Adv. Math. {\bf 123}, no. 1 (1996), 16-77.

\bibitem[NS95]{NS95}
	M. Noumi, T. Sugitani,
	{\it Quantum symmetric spaces and related q-orthogonal polynomials}, 
	In: Group Theoretical Methods in Physics (ICGTMP) (Toyonaka, Japan, 1994), World Sci. Publishing, River Edge, N.J. (1995), 28--40.

%
\bibitem[Prz25]{Prz23}
T. Prze\'zdziecki, {\it Drinfeld rational fractions for affine Kac--Moody quantum symmetric pairs}, Selecta Math. (N.S.) {\bf 31} (2025), no.~5, Paper No. 87.

%
%
	
\bibitem[RT90]{RT90}
N. Reshetikhin, V. Turaev, {\it Ribbon graphs and their invariants derived from quantum groups}, Comm. Math. Phys. {\bf 127} (1990), no.~1, 1--26. \MR{1036112}

	
\bibitem[RV18]{RV18}
	V. Regelskis, B. Vlaar, 
	{\it Solutions of the $U_q(\wh{\fksl}_N)$ reflection equations},
	J. Phys. A: Math. Theor. \textbf{51}, no.~34 (2018), 345204.

\bibitem[RV20]{RV20}
	V. Regelskis, B. Vlaar, 
	{\it Quasitriangular coideal subalgebras of {$U_q(\mathfrak g)$} in terms of generalized Satake diagrams},
	Bull. Lond. Math. Soc. \textbf{52}, no.~4 (2020), 693--715. \MR{4171396}
	
\bibitem[RV22]{RV22}
	V. Regelskis, B. Vlaar, 
	{\it Pseudo-symmetric pairs for {K}ac-{M}oody algebras},
	Hypergeometry, integrability and Lie theory, 155--203, Contemp. Math. {\bf 780}, Amer. Math. Soc. (2022).
	
%
%


\bibitem[Skl88]{Skl88}
	E. Sklyanin, 
	{\it Boundary conditions for integrable quantum systems},
	J. Phys. A: Math. Gen. \textbf{21}, no.~10 (1988), 2375.

\bibitem[ST09]{ST09}
	N. Snyder, P. Tingley,
	{\it The half-twist for $U_q(\mathfrak g)$ representations}.
	Algebra \& Number Theory. {\bf 3}, no.~7 (2009), 809--834.

%
%
%
%
\bibitem[tDHO98]{tDHO98}
	T. tom Dieck, R. H\"{a}ring-Oldenburg,
	{\it Quantum groups and cylinder braiding},
	Forum Math. {\bf 10}, no.~5 (1998), 619--639.

%
\bibitem[Vla15]{Vla15}
	B. Vlaar,
	{\it Boundary transfer matrices and boundary quantum KZ equations},
	J. Math. Phys. {\bf 56}, no. 7 (2015), 071705.



%
\bibitem[Wa24]{Wa24}
	H. Watanabe,
	{\it Crystal bases of modified $\imath$quantum groups of certain quasi-split types},
	Algebras and Represent. Theory {\bf 27}, no. 1 (2024), 1--76.

\bibitem[WZh23]{WZh23}
	W. Wang, W. Zhang,
	{\it An intrinsic approach to relative braid group symmetries on $\imath$quantum groups},
	Proc. London Math. Soc. {\bf 127}, no.~5 (2023), 1338--1423.

\bibitem[Xu23]{Xu23}
	X. Xu,
	{\it Stokes Phenomenon and Reflection Equations},
	Commun. Math. Phys. {\bf 398} (2023), 353--373.

\bibitem[YNZ06]{YNZ06}
	W.-L. Yang, R. Nepomechie, Y.-Z. Zhang,
	{\it Q-operator and T-Q relation from the fusion hierarchy},
	Phys. Lett. B {\bf 663}, no. 4-5 (2006), 664--670.

\bibitem[Zh22]{Zh22}
	W. Zhang, 
	{\it A Drinfeld-type presentation of affine $\imath$quantum groups II: split BCFG type},
	Lett. Math. Phys., {\bf 112},  no. 5 (2022), p89.

\end{thebibliography}
\end{document}